\documentclass[12pt]{amsart}
\usepackage{graphicx}
\newtheorem{theorem}{Theorem}[section]
\newtheorem{proposition}[theorem]{Proposition}
\newtheorem{lemma}[theorem]{Lemma}

\theoremstyle{definition}
\newtheorem{definition}[theorem]{Definition}

\newcommand{\bdd}{\mbox{$\partial$}}

\theoremstyle{remark}

\numberwithin{equation}{section}

\begin{document}

\title[$1$-genus $1$-bridge knots] {
Genus two Heegaard splittings
of exteriors of \\$1$-genus $1$-bridge knots
}

\author{Hiroshi Goda}
\author{Chuichiro Hayashi}

\date{\today}

\thanks{The first and second authors are partially supported
by Grant-in-Aid for Scientific Research, (No. 21540071 and 18540100),
Ministry of Education, Science, Sports and Culture.}

\begin{abstract}
A knot $K$ in a closed connected orientable $3$-manifold $M$ 
is called a $1$-genus $1$-bridge knot if $(M,K)$ has a 
splitting into two pairs of a solid torus $V_i$ $(i=1,2)$ 
and a boundary parallel arc in it.
The splitting induces a genus two Heegaard splitting of the exterior 
of $K$ naturally, i.e., $K$ has an unknotting tunnel. 
However the converse is not true in general. 
Then we study such general case in this paper. 
One of the conclusions is that the unknotting tunnel 
may be levelled with the torus $\partial V_1=\partial V_2$.
\end{abstract}

\maketitle


\section{Introduction}

A properly embedded arc $t$ in a solid torus $V$
is called {\it trivial\/}
if it is boundary parallel,
that is, there is a disk $C$ embedded in $V$
such that $t \subset \bdd C$
and $C \cap \bdd V = \text{cl}\,(\bdd C - t)$.
We call such a disk a {\it canceling disk\/} of the trivial arc $t$.
In this paper, we denote by $M$ a closed connected orientable $3$-manifold.
Let $K$ be a knot in $M$.
We call $K$ a {\it $1$-genus $1$-bridge knot\/} in $M$
if $M$ is a union of two solid tori $V_1$ and $V_2$
glued along their boundary tori $\bdd V_1$ and $\bdd V_2$
and if $K$ intersects each solid torus $V_i$ in a trivial arc $t_i$
for $i=1$ and $2$.
The splitting $(M, K) = (V_1, t_1) \cup_{H_{1}} (V_2, t_2)$
is called a {\it $1$-genus $1$-bridge splitting\/} of $(M, K)$,
where $H_{1} = V_1 \cap V_2 = \bdd V_1 = \bdd V_2$, the torus.
We call also the splitting torus $H_{1}$
a {\it $1$-genus $1$-bridge splitting}.
We say $(1,1)$-knots and $(1,1)$-splitting for short.

It is well-known
that all torus knots and $2$-bridge knots
in the $3$-sphere $S^3$ are $1$-genus $1$-bridge knots.
$(1,1)$-splittings of these knots are studied
by K. Morimoto \cite{M}, and T. Kobayashi - O. Saeki \cite{KS}.
In \cite{Hy3}, the second author studied $(1,1)$-splittings 
of $1$-genus $1$-bridge knots.

We recall the definition of a $(2,0)$-splitting.
Let $W$ be a handlebody,
and $K$ a knot in $\text{int}\,W$.
We say $K$ is a {\it core\/} in $W$
if there are a disk $D$ and an annulus $A$
such that $D$ is properly embedded in $W$
and intersects $K$ transversely in a single point
and that $A$ is embedded in $W$
with $K \subset \bdd A$ and $A \cap \bdd W = \bdd A-K$.
We say that the pair $(M, K)$ admits a {\it $(2,0)$-splitting\/}
if $M$ is a union of two handlebodies of genus two,
say $W_1$ and $W_2$,
glued along
$\bdd W_1$ and $\bdd W_2$
and
if $K$ forms a core in $W_1$.
The closed surface $H_{2} = \bdd W_1 = \bdd W_2 = W_1 \cap W_2$
gives the splitting $(M, K) = (W_1, K) \cup_{H_{2}} (W_2, \emptyset)$
and is called a {\it $(2,0)$-splitting surface\/}
or
{\it a $(2,0)$-splitting\/} for short.
It is easy to see
that $\text{cl}\,(W_1 - N(K))$ is a compression body
homeomorphic to a union of (a torus)\,$\times [0,1]$
and a $1$-handle
which has attaching disks in (a torus)\,$\times \{ 1 \}$.
Hence
$H_{2} = \bdd W_1 = \bdd W_2$
gives a genus two Heegaard splitting
of the exterior of
$K$.

 A $(1,1)$-knot admits a $(2,0)$-splitting as follows.
 Let $(M,K) = (V_1, t_1) \cup_{H_{1}} (V_2, t_2)$ be a $(1,1)$-splitting.
 We take a regular neighborhood $N(t_2)$ of the arc $t_2$ in $V_2$.
 Then $(M, K)=(V_1 \cup N(t_2), K)
\cup (\text{cl}\,(V_2-N(t_2)), \emptyset)$
is a $(2,0)$-splitting.
 If we take a regular neighborhood $N(t_1)$
of the arc $t_1$ in $W_1$,
then $(M, K)=(V_2 \cup N(t_1), K)
\cup (\text{cl}\,(V_1-N(t_1)), \emptyset)$
is another $(2,0)$-splitting.
 These are $(2,0)$-splittings
naturally derived from the $(1,1)$-splitting.
 Such $(2,0)$-splittings are characterized
in the following manner.
 A $(2,0)$-splitting $(M, K) = (W_1, K) \cup_H (W_2, \emptyset)$
is {\it meridionally stabilized\/}
if there is a disk $D_i$ properly embedded in $W_i$ for $i=1$ and $2$
such that
$\bdd D_1$ and $\bdd D_2$ intersect
each other transversely in a single point in $H = \bdd W_1 = \bdd W_2$
and that $D_1$ intersects
$K$ transversely
in a single point.
A $(2,0)$-splitting
$(M, K)=(V_i \cup N(t_j), K)
\cup (\text{cl}\,(V_j-N(t_j)), \emptyset)$,
which is derived from a $(1,1)$-splitting
$(M,K) = (V_1, t_1) \cup (V_2, t_2)$,
is meridionally stabilized
since we can take the disk $D_1$
to be a meridian disk of the arc $t_j$ in $N(t_j)$,
and the disk $D_2$ to be a canceling disk of the arc $t_j$.
Conversely,
we can obtain a $(1,1)$-splitting torus
by compressing the meridionally stabilized $(2,0)$-splitting surface
along
$D_1$.

Here is a question:
Is any $(2,0)$-splitting of a $(1,1)$-knot meridionally stabilized ?
It was pointed out by K. Morimoto that
every torus knot has
only a single isotopy class of $(1,1)$-splitting torus,
which is an easy corollary of Theorem 3 in \cite{M}
and the uniqueness of genus one Heegaard splitting
(see \cite{Bo}, \cite{BoO} and \cite{RS}).
If all the $(2,0)$-splitting were meridionally stabilized
for some torus knot,
then the torus knot exterior
would have at most two genus two Heegaard splittings
derived from the unique $(1,1)$-splitting.
 However, there is a torus knot
such that its exterior has three genus two Heegaard splittings
as shown in \cite{BRZ}
by Z. Boileau, M. Rost and H. Zieschang.
 Hence such a torus knot has a $(2,0)$-splitting
which is not meridionally stabilized.
(This $(2,0)$-splitting is derived from the unknotting tunnel
such that it can be isotoped into the torus
on which the torus knot lies.)

Let $W$ be a handlebody of genus two,
and $K$ a core in $W$.
A graph $\gamma$ embedded in $\text{int}\,W$
is called a {\it spine\/} of $(W, K)$
if $\gamma \cap K = \bdd \gamma$
and $W$ collapses to $K \cup \gamma$.
($\bdd \gamma$ denotes the union of the vertices of valency one in $\gamma$.)
An essential loop $l$ in the boundary of a handlebody $W$
is called a {\it meridian\/}
if it bounds a disk in $W$.
A loop $l'$ in the boundary of a handlebody $W'$
is called a {\it longitude\/}
if there is a meridian loop of $W'$
such that it intersects $l'$ transversely in a single point.

We say that a $(1,1)$-splitting
$(M, K) = (V_1, t_1) \cup_{H_1} (V_2, t_2)$
admits a {\it satellite diagram\/}
if there is an essential simple loop $l$ on the torus $H_1$
such that the arcs $t_1$ and $t_2$ have canceling disks
which are disjoint from $l$.
We call $l$ the {\it slope\/} of the satellite diagram.
We say that the slope of the satellite diagram is
{\it meridional\/} (resp. {\it longitudinal\/})
if it is meridional (resp. longitudinal)
on $\partial V_1$ or $\partial V_2$.

\begin{theorem}\label{thm:main}
Let $K$ be a knot in the $3$-sphere $S^3$.
Suppose that there are given
a $(1,1)$-splitting
$(S^3, K) = (V_1, t_1) \cup_{H_1} (V_2, t_2)$
and a $(2,0)$-splitting
$(S^3, K) = (W_1, K) \cup_{H_2} (W_2, \emptyset)$.
Then at least one of the following conditions holds.
\begin{enumerate}
\renewcommand{\labelenumi}{(\theenumi)}
\item
The $(2,0)$-splitting $H_2$ is meridionally stabilized.
\item
There is an arc $\gamma$
which forms a spine of $(W_1, K)$
and is isotopic into the torus $H_1$.
Moreover, we can take $\gamma$
so that
there is a canceling disk $C_i$ of the arc $t_i$ in $(V_i, t_i)$
with $\bdd C_i \cap \gamma = \partial \gamma = \partial t_i$
for $i=1$ or $2$.
\item
There is an essential separating disk $D_2$ in $W_2$,
and an arc $\alpha$ in $W_1$
such that $\alpha \cap K$ is one of the endpoints $\bdd \alpha$,
and $\alpha \cap \bdd W_1$ is the other endpoint $p$ of $\alpha$
and that $D_2$ cuts off a solid tours $U_1$ from $W_2$ with
$p \in \bdd U_1$
and with the torus $\partial N(U_1 \cup \alpha)$
isotopic to
$H_1$ in $(M, K)$.
$($See Figure \ref{fig:Case3}.$)$
\item
The $(1,1)$-splitting $H_1$ admits a satellite diagram
of a longitudinal slope.
\end{enumerate}
\end{theorem}

\begin{figure}[htbp]
\centering
\includegraphics[width=.6\textwidth]{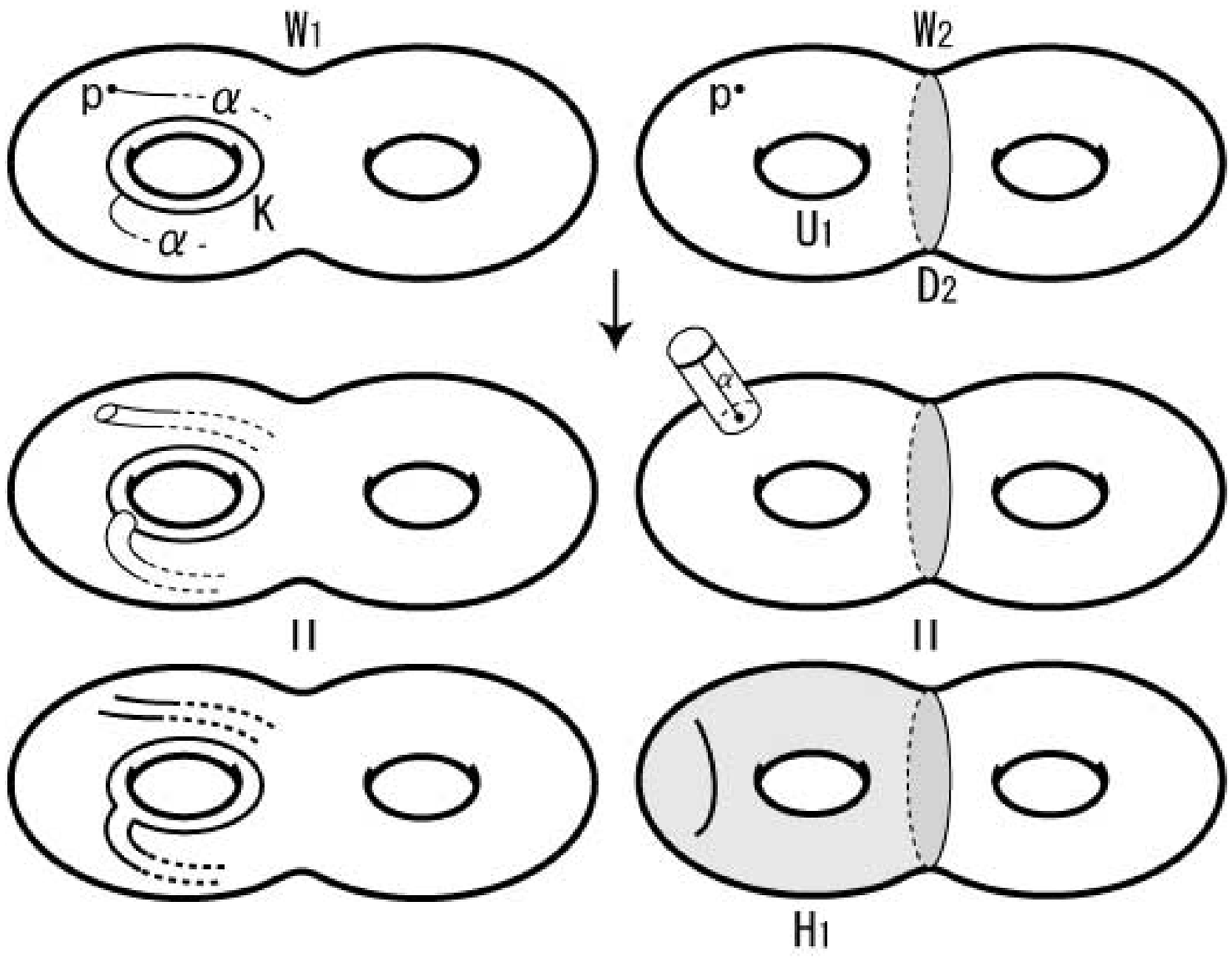}
\caption{}
\label{fig:Case3}
\end{figure}

Note that the conclusion (2) contains torus knots, and
$K' = \gamma \cup t_i$ forms a torus knot
and the complementary arc $t_j = \text{cl}\,(K - t_i)$
forms an unknotting tunnel for $K'$
when $t_i$ is isotoped into the $(1,1)$-splitting torus $H_1$
along the canceling disk $C_i$.
But we cannot apply the classification of unknotting tunnel
of torus knot given in \cite{BRZ}
because we cannot slide an endpoint of the arc $t_j$
beyond the other endpoint of $t_j$.
H.J. Song 
informed us that there are concrete examples
of hyperbolic knots and their unknotting tunnels satisfying the
conclusion (2). See Section \ref{sec:example}.
The conclusion (3) resembles the `dual tunnel' case as described 
in (\cite{MS},(1,1)), 
but the arc $\alpha$ may be knotted or linked with $K$. 
A knot in the conclusion (4) can be obtained from 
a component of a 2-bridge link by a $1/n$-Dehn surgery 
on the other component, that is, twisting. 

It is shown in Theorem III in \cite{Hy2}
that a $1$-genus $1$-bridge splitting of a satellite knot
has a satellite diagram
of a non-meridional and non-longitudinal slope.
When the slope $l$ is longitudinal on $\bdd V_1$,
the boundary torus of the regular neighborhood
of $(H_1 -N(l)) \cup C_2$ also gives a $(1,1)$-splitting,
where $C_2$ is a canceling disk of $t_2$
with $C_2 \cap l = \emptyset$.


\medskip

\noindent {\bf Question.}
(1) Is there an example which realizes the conclusion (3) ?\\
(2) How does an unknotting tunnel of a knot in the conclusion (4) behave ?

\medskip

We note that every unknotting tunnel of a tunnel number one knot
in $S^3$ may be slid and isotoped to lie entirely in its
minimal bridge sphere \cite{GST}.

In the rest of this section,
we recall the main machinery
and give a more precise statement of the above result
for knots in the $3$-sphere and lens spaces.

In \cite{RS},
H. Rubinstein and M. Scharlemann showed
that two Heegaard splitting surfaces
of a closed connected orientable $3$-manifold
can be isotoped
so that they intersect each other
in non-empty collection of essential loops
if they are not weakly reducible.
T. Kobayashi and O. Saeki studied in \cite{KS}
a variation of it, 
where $3$-manifolds contain links.
We recall one of their results.

Let $X$ be a compact orientable $3$-manifold,
and $T$ a compact $1$-manifold properly embedded in $X$.
For $i=1$ and $2$,
let $F_i$ be either a $2$-submanifold of $\bdd X$
or a compact orientable $2$-manifold
which is properly embedded in $X$ and is transverse to $T$.
Suppose that $T \cap \bdd F_i = \emptyset$ for $i=1$ and $2$.
$F_1$ is said to be
{\it $T$-compressible\/} in $(X, T)$
if there is a disk $D_1$ embedded in $X$
such that $D_1 \cap F_1 = \bdd D_1$,
that $D_1$ is disjoint from $T$
and that $\bdd D_1$ does not bound a disk in $F_1-T$.
We call such a disk a {\it $T$-compressing disk}.
$F_2$ is said to be
{\it meridionally compressible\/} in $(X, T)$
if there is a disk $D_2$ embedded in $X$
such that $D_2 \cap F_2 = \bdd D_2$,
that $D_2$ intersects $T$ transversely in a single point
and that $\bdd D_2$ does not bound a disk
which intersects $T$ in a single point in $F_2$.
 We call such a disk a {\it meridionally compressing disk}.

A $(1,1)$-splitting $(M, K)=(V_1, t_1) \cup_H (V_2, t_2)$
is called {\it weakly $K$-reducible\/}
if there is a $t_i$-compressing or meridionally compressing disk $D_i$
of $H = \bdd V_i$ in $(V_i, t_i)$ for $i=1$ and $2$
such that $\bdd D_1 \cap \bdd D_2 = \emptyset$.
 A $(1,1)$-splitting is called {\it strongly $K$-irreducible\/}
if it is not weakly $K$-reducible.
$(1,1)$-knots which admit a weakly $K$-reducible $(1,1)$-splitting
are characterized in Lemma 3.2 in \cite{Hy3}.
We recall it in Proposition \ref{lem:11weak} in Section 4.

A $(2,0)$-splitting $(M, K) = (W_1, K) \cup_{H_2} (W_2, \emptyset)$
is called {\it weakly $K$-reducible\/}
if there is a $K$-compressing or meridionally compressing disk $D_1$
of $H_2 = \bdd W_1$ in $(W_1, K)$
and a compressing disk $D_2$ of $H_2 = \bdd W_2$ in $W_2$
such that $\bdd D_1 \cap \bdd D_2 = \emptyset$.
$(2,0)$-knots which admit a weakly $K$-reducible $(2,0)$-splitting
are characterized in Proposition 2.14 in \cite{GHY}.
We recall it in Proposition \ref{prop:20weak} in Section 4.
There we find
that a meridionally stabilized $(2,0)$-splitting
is weakly $K$-reducible.

Suppose that a compact orientable $2$-manifold $F$
is properly embedded in $X$
so that it is transverse to $T$
and $\bdd F$ is disjoint from $T$.
 A loop in $F-T$ is called {\it $T$-inessential\/}
if either it bounds a disk in $F-T$,
or it bounds a disk $D$ in $F$
such that $D$ intersects $T$ in a single point.
 Otherwise, it is {\it $T$-essential}.

\begin{theorem}[T.Kobayashi and O.Saeki \cite{KS}]\label{thm:KobayashiSaeki}
Suppose that
$L$ is a link in $M$ that has a $2$-fold branched covering
with branch set $L$.
 Let $(M,L)=(W_{i1}, t_{i1}) \cup_{H_i} (W_{i2}, t_{i2})$
be a $g_i$-genus $n_i$-bridge splitting for $i=1$ and $2$.
 Suppose that every component of $L$ intersects $H_i$ for $i=1$ and $2$,
and that these splittings are not weakly $L$-reducible.
 Then we can isotope $H_1$ or $H_2$ in $(M, L)$
so that they intersect each other
in non-empty collection of finite number of loops
which are $L$-essential both in $H_1$ and in $H_2$.
\end{theorem}

We do not recall the definition
of $g_i$-genus $n_i$-bridge splittings,
but $1$-genus $1$-bridge splittings are special cases of them.
The above theorem does not work for $(2,0)$-splittings
because the knot does not intersect the splitting surface.
But a similar argument to the proof of the above theorem
works for a pair of a $(1,1)$-splitting and a $(2,0)$-splitting.
However, the above theorem contains a technical condition
on existence of $2$-fold branched covering.
For example,
the projective space ${\mathbb R}P^3$,
which is homeomorphic to the lens space $L(2,1)$,
does not have a $2$-fold branched covering
with a core knot being a branch set, 
where the exterior of a core knot is homeomorphic to a solid torus.
Of course,
the $3$-sphere $S^3$ has a $2$-fold branched covering
with any knot being a branch set.

\medskip

\noindent{\bf Question.}
Can we reduce the condition that $M$ has a $2$-fold branched covering
with branch set $L$ in the above theorem ?

\medskip
In this paper, we begin with the situation
that a $(1,1)$-splitting and a $(2,0)$-splitting intersect each other
in non-empty collection of finite number of $K$-essential loops.
We will also use this condition
when we apply Proposition \ref{thm:koba}.
The authors expect
that this proposition, and hence the following theorems hold
without this technical condition.

We say that $K$ is a {\it torus knot\/}
if $K$ can be isotoped into a torus
which gives a genus one Heegaard splitting of $M$.
We call $K$ a {\it satellite knot\/}
if the exterior $E(K)=\text{cl}\,(M-N(K))$
contains an incompressible torus $T$ 
which is not parallel to $\bdd E(K)$.
The torus $T$ may not bound a solid torus in $M$.

\begin{theorem}\label{thm:general}
Let $M$ be the $3$-sphere or a lens space $($other than $S^2 \times S^1)$,
and $K$ a knot in $M$.
Let $(V_1, t_1) \cup_{H_1} (V_2, t_2)$ and  $(W_1, K) \cup_{H_2} (W_2, \emptyset)$
be a $(1,1)$-splitting and $(2,0)$-splitting of $(M,K)$.
Suppose that the splitting surfaces $H_1$ and $H_2$
intersect each other in $\ell$ loops
which are $K$-essential both in $H_1$ and in $H_2$.
\begin{enumerate}
\item[(1)] 
If $\ell \ge 4$, 
then at least one of the five conditions $(a), (b), (c), (d)$ and $(e)$ below holds.
\item[(2)] 
If $\ell=3$, 
then at least one of the six conditions $(a), (b), (c), (d), (e)$ and $(f)$ below holds.
\item[(3)] 
We assume that $M$ has a $2$-fold branched cover with branch set $K$.
\begin{enumerate}
\item[(3-1)] 
If $\ell=2$, 
then at least one of 
the six conditions $(a), (c), (d), (e), (f)$ and $(h)$ below holds.
\item[(3-2)] 
If $\ell=1$, 
then at least one of the conditions $(c), (d)$ and $(g)$ below holds.
\end{enumerate}
\end{enumerate}
\begin{enumerate}
\item[(a)]
 We can isotope $H_1$ and $H_2$ in $(M, K)$
so that they intersect each other
in non-empty collection of smaller number of loops
which are $K$-essential both in $H_1$ and in $H_2$.
\item[(b)]
 The $(1,1)$-splitting $H_1$ is weakly $K$-reducible.
\item[(c)]
 The $(2,0)$-splitting $H_2$ is weakly $K$-reducible.
\item[(d)]
 The knot $K$ is a torus knot.
\item[(e)]
 The knot $K$ is a satellite knot.
\item[(f)]
The $(1,1)$-splitting $H_1$ admits a satellite diagram 
of a longitudinal slope. 
\item[(g)]
 There is an arc $\gamma$
which forms a spine of $(W_1, K)$
and is isotopic into the torus $H_1$.
 Moreover, we can take $\gamma$
so that
there is a canceling disk $C_i$ of the arc $t_i$ in $(V_i, t_i)$
with $\bdd C_i \cap \gamma = \partial \gamma = \partial t_i$
for $i=1$ or $2$.
\item[(h)]
 There is an essential separating disk $D_2$ in $W_2$,
and an arc $\alpha$ in $W_1$
such that $\alpha \cap K$ is one of the endpoints $\bdd \alpha$,
and 
$\alpha \cap \partial W_1$ 
is the other endpoint, say $p$, of $\alpha$
and that $D_2$ cuts off a solid tours $U_1$ from $W_2$
with $p \in \bdd U_1$
and with the torus $\partial N(U_1 \cup \alpha)$
isotopic to the $(1,1)$-splitting torus $H_1$ in $(M, K)$.
$($See Figure \ref{fig:Case3}.$)$
\end{enumerate}
\end{theorem}

If the case $(b)$ or $(c)$ occurs, 
the knot is in a well-studied class
or $H_2$ becomes meridionally stabilized
by Proposition \ref{lem:11weak}, \ref{prop:KReducible} 
and \ref{prop:20weak}.
(1) in this theorem will be proved in Sections 5, 7 and 8.
Section 6 is devoted to prove the case (2).
(3-2) will be proved in Section 9
which can be read without reading 
Sections 5 through 8.
The conclusion (h) corresponding to (3) 
in Theorem \ref{thm:main}
appears in only (3-1), 
which will be studied in the sequel to
this paper \cite{GH}.  
See also Theorem 1.1 in \cite{GH}.


Theorem \ref{thm:main} is shown via 
Theorem \ref{thm:general}, 
Propositions \ref{lem:11weak} and \ref{prop:20weak}
together with the results in \cite{BRZ},  \cite{K1}. \cite{MS}, 
\cite{N}, and \cite{S}.  
A $(2,0)$-splitting for a torus knot in $S^3$
satisfies conclusion (1) or (2) of Theorem \ref{thm:main},
and conclusion (1) holds
for that for a satellite knot or a $2$-bridge knot.
A knot with a $(1,1)$-splitting in $S^3$ is prime by \cite{N} and \cite{S}.

We should note that almost all parts of this paper were written in 2001. 
We were motivated to finish writing by the works of 
S. Cho-D. McCullough \cite{CM1,CM2}, K. Ishihara \cite{I}, 
J. Johnson \cite{johnson}, 
Y. Koda \cite{koda}, and M. Scharlemann-M. Tomova \cite{ST}.

The authors would like to thank
Professor Tsuyoshi Kobayashi,
Professor Kanji Morimoto and
Professor Makoto Sakuma
for helpful comments.
The last version of this paper has been done 
by the virtue of Dr. Kai Ishihara and Dr. Yuya Koda.  
The authors thank them for valuable comments and advices. 
They also thank the referee for several comments.


\section{Surfaces in a solid torus with a trivial arc}

Throughout this section, 
we study properties of surfaces
in a solid torus $V$ 
with a trivial arc $t$ properly embedded in $V$.
Standard cut and paste arguments show the next two lemmas.
We omit the proofs.

\begin{lemma}\label{lem:solid1}
 If $D$ is a $t$-compressing disk of $\bdd V$,
then either
\begin{enumerate}
\renewcommand{\labelenumi}{(\theenumi)}
\item
$D$ is a meridian disk of $V$ and is disjoint from $t$, or
\item
$D$ is a $\bdd$-parallel disk
which separates $V$ into a solid torus
and a 3-ball containing $t$ $($see Figure \ref{fig:t-comp}\/$)$.
\end{enumerate}
\end{lemma}

\begin{figure}[htbp]
\centering
\includegraphics[width=.6\textwidth]{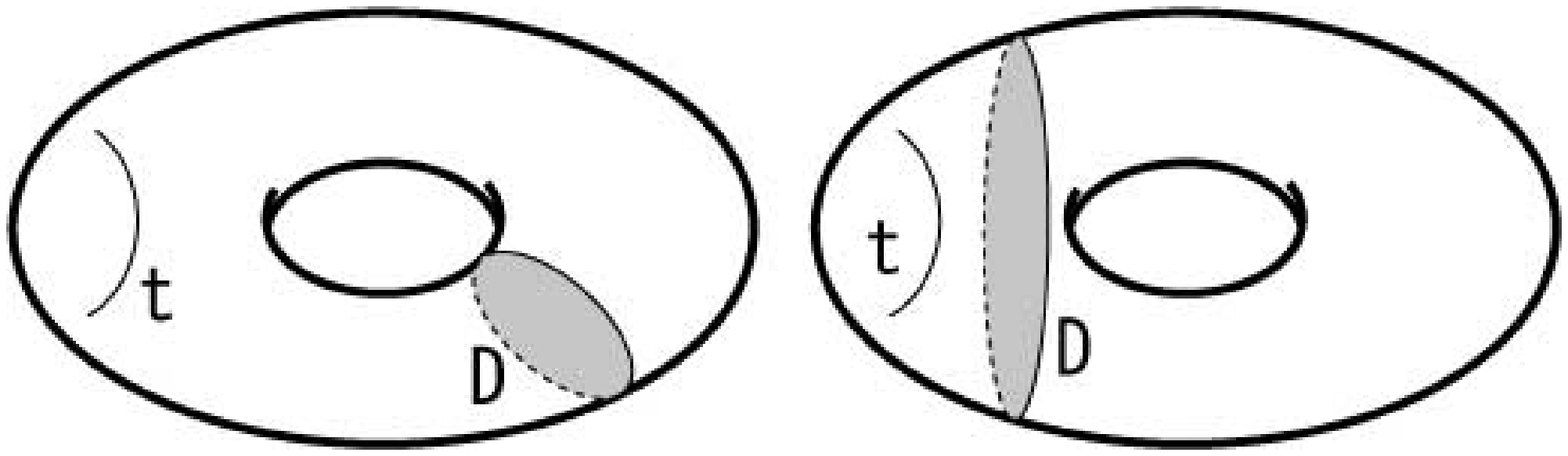}
\caption{}
\label{fig:t-comp}
\end{figure}

\begin{lemma} \label{lem:solid2}
 If $D$ is a meridionally compressing disk of $\bdd V$,
then $D$ is a meridian disk in $V$
and intersects $t$ transversely at a single point $($see Figure \ref{fig:meri}\/$)$.
\end{lemma}

%

\begin{figure}[htbp]
\centering
\includegraphics[width=.3\textwidth]{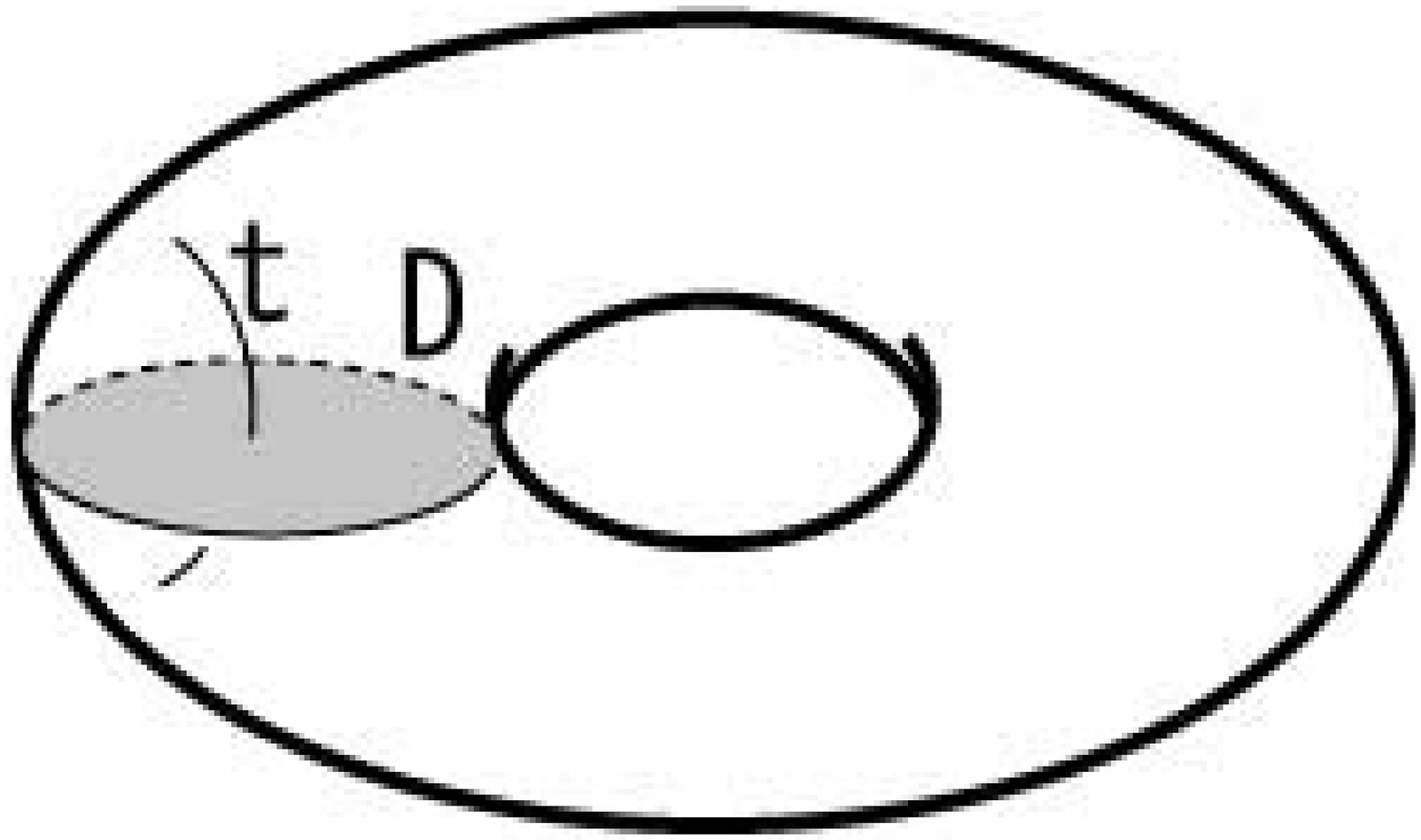}
\caption{}
\label{fig:meri}
\end{figure}

\begin{definition}($T$-$\bdd$-compressible)\label{def:bdd-comp}
 Let $X$ be an orientable $3$-manifold,
and $T$ a compact $1$-manifold properly embedded in $X$.
 Let $F$ be a compact orientable $2$-manifold
properly embedded in $X$.
 Suppose that $\bdd F$ is disjoint from $T$
and that $T$ is transverse to $F$.
 We say that $F$ is {\it $T$-$\bdd$-compressible\/} in $(X, T)$
if there is a disk $D$ embedded in $X$
satisfying all of the following conditions:
\begin{enumerate}
\renewcommand{\labelenumi}{(\theenumi)}
\item
$D$ is disjoint from $T$;
\item
$D \cap (F \cup \bdd X) = \bdd D$;
\item
$D \cap F$ is an essential arc properly embedded in $F-T$;
\item
$\bdd D \cap \bdd X$ is an essential arc
in the surface obtained from $\bdd X - T$
by cutting along $\bdd F$.
\end{enumerate}
 We call such a disk $D$ a {\it $T$-$\bdd$-compressing disk\/} of $F$.
 When there is not such a disk,
we say that $F$ is {\it $T$-$\bdd$-incompressible\/} in $(X, T)$.
\end{definition}

Remark.
 In the usual definition, the above condition (4) is omitted,
but we add this in this paper as in \cite{GHY} and \cite{Hy3}.
Note that this definition 
is equivalent to the usual one 
when $F$ is $T$-incompressible.

\begin{lemma}[Lemma 2.10 in \cite{Hy3}]\label{lem:H}
 Let $F$ be a compact orientable 2-manifold properly embedded in $V$
so that $F$ is transverse to $t$.
 Suppose that $F$ is $t$-incompressible and $t$-$\bdd$-incompressible
in $(V, t)$.
 Then $F$ is a union of finitely many surfaces of types $(1)\sim (6)$ below:
\begin{enumerate}
\renewcommand{\labelenumi}{(\theenumi)}
\item
a 2-sphere disjoint from $t$;
\item
a 2-sphere intersecting $t$ transversely in two points;
\item
a meridian disk of $V$ disjoint from $t$;
\item
a meridian disk of $V$ intersecting $t$ transversely in a single point;
\item
a peripheral disk disjoint from $t$;
\item
a peripheral disk intersecting $t$ transversely in a single point.
\end{enumerate}
\end{lemma}

\begin{lemma}\label{lem:annulus}
Let $A$ be an annulus properly embedded in $(V,t)$
such that a component of $\bdd A$ is essential in $\bdd V$
(ignoring $\bdd t$)
and the other component of $\bdd A$ bounds a disk $Q$ in $\bdd V$.
 Suppose that the annulus $A$ is disjoint from
$t$
and that
$Q$ contains the two endpoints of $t$.
 Then $A$ is $t$-compressible.
\end{lemma}

\begin{proof}
By Lemma \ref{lem:H},
$A$ is $t$-compressible
or $t$-$\bdd$-compressible in $(V, t)$.
In the latter case,
a $t$-$\bdd$-compression on $A$
yields
a meridian disk $D$ of
$V$.
$D$ is disjoint from $A$ after an adequate small isotopy.
 Then $\bdd D$ and 
a component of $\bdd A$ cobounds on $\partial V$ an annulus,
say $R$, which does not contain
$Q$.
Then $D\cup R$ 
gives a $t$-compressing disk of $A$.
\end{proof}

\begin{lemma}\label{lem:pants}
 Let $P$ be a disk with two holes properly embedded in $V$
so that $P$ is disjoint from $t$.
Suppose that every component of $\bdd P$ is $t$-essential in $\bdd V$.
Then, either
\begin{enumerate}
\renewcommand{\labelenumi}{(\theenumi)}
\item
 $P$ is $t$-compressible or meridionally compressible in $(V,t)$, or
\item
 there exists a $t$-$\bdd$-compressing disk $D$ of $P$ in $(V,t)$
such that the arc $\bdd D\cap P$ connects two distinct
components of $\bdd P$
and that every component of $\bdd P'$ is $t$-essential in $\bdd V$,
where $P'$ is the annulus
obtained from $P$ by $t$-$\bdd$-compression along $D$.
\end{enumerate}
 Moreover,
if each component of the loops $\bdd P$ is essential in $\bdd V$
$($ignoring the endpoints $\bdd t\/)$,
then $P$ is $t$-compressible or meridionally compressible in $(V, t)$.
\end{lemma}

\begin{proof}
 By Lemma \ref{lem:H},
$P$ is $t$-compressible or $t$-$\bdd$-compressible in $(V,t)$.
In the latter case, 
let $Q$ be a $t$-$\bdd$-compressing disk of $P$ in $(V,t)$.
Set 
$\beta=\bdd Q\cap\bdd V$.
 We have five cases:
(i) $\beta$ is in an annulus component $A$ of $\bdd V-\bdd P$
      such that $A$ contains at most one of the endpoints $\bdd t$;
(ii) $\beta$ is in an annulus component $A'$ of $\bdd V-\bdd P$
     such that $A'$ contains the two endpoints $\bdd t$;
(iii) $\beta$ is in  a disk component $R$ of $\bdd V-\bdd P$
     such that $R$ contains the two endpoints $\bdd t$;
(iv) $\beta$ is in a torus with 1-hole component $U$ of $\bdd V-\bdd P$
     with $U \cap \bdd t = \emptyset$; and 
(v) $\beta$ is in a disk with 2-holes component $Z$ of $\bdd V-\bdd P$
     with $Z \cap \bdd t = \emptyset$.
 Note that (i) or (ii) occurs
when each component of $\bdd P$ is essential in $\bdd V$
(ignoring $\bdd t$).

Cases (i) and (iii).
 `$\bdd$-compression' on a copy of $A$ or $R$ along $Q$
yields a $t$-compressing disk or meridionally compressing disk of $P$.
 Hence we obtain the conclusion (1).
 Note that $\beta$ separates $R$ into two disks
each of which intersects $t$ at a single point
by Definition \ref{def:bdd-comp}.

Case (ii).
 The three boundary loops of $\bdd P$ are essential in $\bdd V$
(ignoring $\bdd t$).
If $\beta$ is essential in $A'$ (ignoring $\bdd t$),
then the annulus obtained from $P$ by $t$-$\bdd$-compression along $Q$
is $t$-compressible by Lemma \ref{lem:annulus}.
Hence $P$ is also $t$-compressible,
and we obtain the conclusion (1).
 Thus we can assume
that $\beta$ is inessential in $A'$
and cuts off a disk from $A'$.
 This disk contains one or two endpoints $\bdd t$
by the condition (4) in Definition \ref{def:bdd-comp}.
 In the former case,
we obtain a meridionally compressing disk of $P$
by `$\bdd$-compression' on a copy of $A'$ along $Q$.
 In the latter case,
$\bdd$-compression on $P$ along $Q$ yields
two annuli, 
one of which has a boundary component
inessential in $\bdd V$ ignoring $\bdd t$,
and is $t$-compressible by Lemma \ref{lem:annulus}.
 This implies that $P$ is $t$-compressible.

Case (iv).
 Every component of $\bdd P$ is inessential in $\bdd V$.
 Performing $\bdd$-compression on $P$ along $Q$,
we obtain two annuli as in Lemma \ref{lem:annulus}.
 The conclusion (1) follows.

\begin{figure}[htbp]
\centering
\includegraphics[width=.7\textwidth]{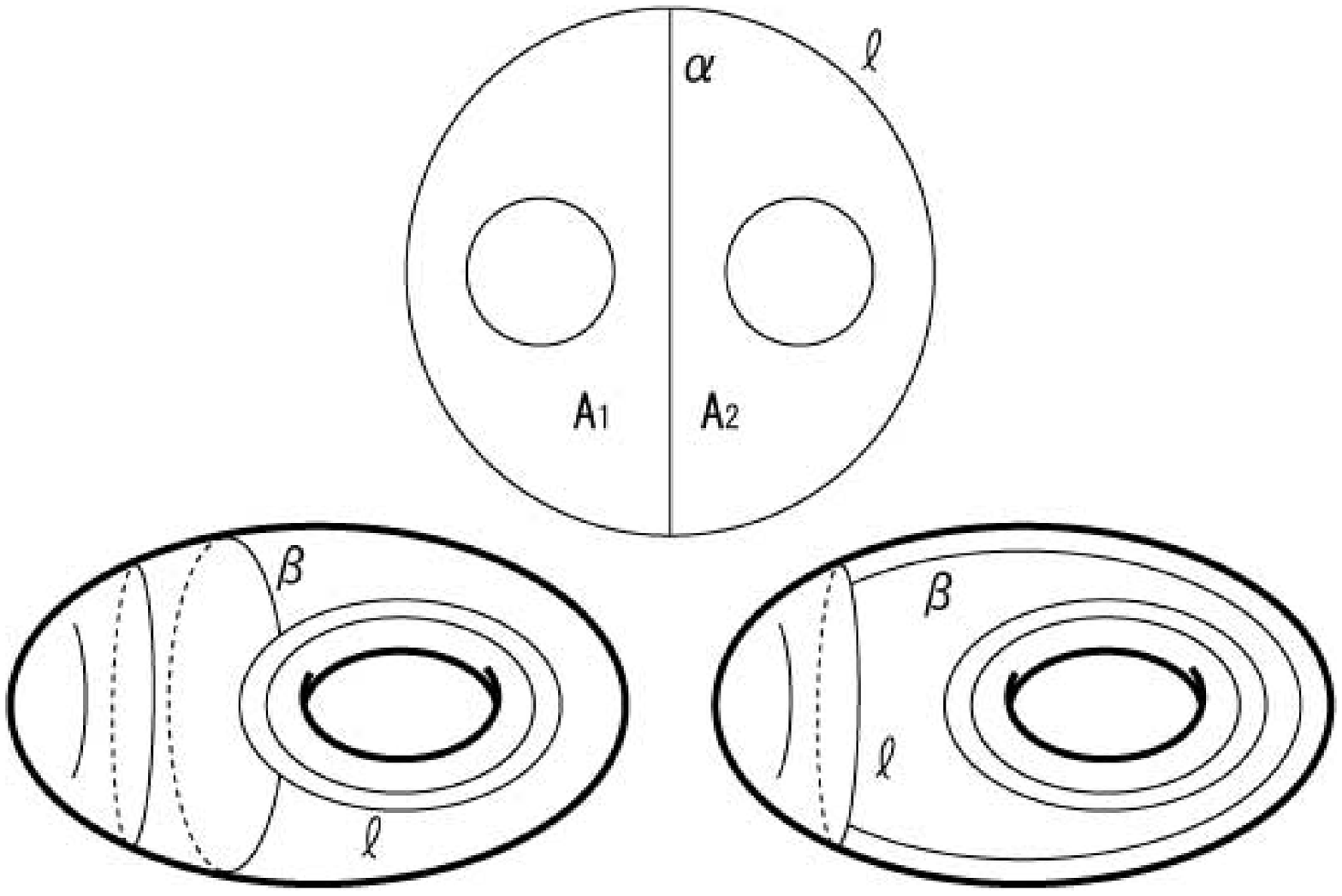}
\caption{}
\label{fig:2-6}
\end{figure}

Case (v).
 If $\beta$ connects two distinct components of $\bdd P$,
then we can obtain an annulus $P'$
from $P$ by $t$-$\bdd$-compression along $Q$
such that $\bdd P'$ is $t$-essential in $\bdd V$.
 This is the conclusion (2).
 Hence we can assume that $\beta$ has its endpoints
in the same component, say $\ell$, of $\bdd P$.
 See Figure \ref{fig:2-6}.
 $t$-$\bdd$-compression on $P$ along $Q$
deforms $P$ into two annuli $A_1$ and $A_2$,
and $\ell$ into two loops,
one of which, say $\ell_1$, is in $\bdd A_1$ 
and the other, say $\ell_2$ in $\bdd A_2$.
 Let $\beta'$ be a `dual' arc of $t$-$\bdd$-compression along $Q$.
 Precisely, $P$ is
a union of $A_1$ and $A_2$ and a tubular neighborhood of $\beta'$ in $\bdd V$
with its interior slightly isotoped into $\text{int}\,V$.
 We say that $P$ is
obtained from $A_1$ and $A_1$ by a band sum along $\beta'$.
 If $A_1 \cup A_2$ is $t$-compressible,
then so is $P$,
and we have the conclusion (1).
 Hence, by Lemma \ref{lem:H},
we can assume that $A_1 \cup A_2$ has a $t$-$\bdd$-compressing disk $D$.
 The arc $\bdd D \cap \bdd V$ connects
two distinct components of $\bdd A_1$ or $\bdd A_2$.
 Hence $\bdd D \cap \bdd V$ is not
in the component of $\bdd V - (\bdd A_1 \cup \bdd A_2)$
which contains the two points $\bdd t$.
 We can isotope $D$ so that it is disjoint from $\beta'$
since $\bdd D \cap \bdd V$ is in an annulus or disk with two holes component
of $\bdd V - (\bdd A_1 \cup \bdd A_2)$.
 Then $D$ gives a $t$-$\bdd$-compressing disk of $P$
after a band sum along $\beta'$.
 Since the arc $\bdd D \cap \bdd V$ connects 
two distinct boundary components of $\bdd A_1$ or $\bdd A_2$, 
it connects two distinct boundary components of $\bdd P$. 
 Applying the arguments on $Q$ in this proof of the lemma to $D$,
we obtain the conclusion (1) or (2).
\end{proof}

\begin{lemma}\label{lem:pants2}
 Let $P$ be a disk with two holes properly embedded in $V$
so that $P$ is disjoint from $t$.
 Suppose that $\bdd P$ is $t$-essential in $\bdd V$,
and that $P$ is $t$-incompressible and meridionally compressible
in $(V, t)$.
 Then a component of $\bdd P$ bounds a meridian disk $D$ in $V$
such that $D$ intersects $t$ transversely in a single point
and that $D \cap P = \bdd D$.
 Moreover,
another component of $\bdd P$ bounds a meridian disk $D'$ in $V$
such that $D'$ intersects $t$ transversely in a single point
and that int\,$D'$ is disjoint from $P$
or intersects $P$ in a single loop.
\end{lemma}

\begin{proof}
 We obtain a disk $R$ and an annulus $A$ from $P$
by meridionally compression on $P$.
 Each of $R$ and $A$ intersects $t$ transversely in a single point.
 Since $\bdd R$ is $t$-essential in $\bdd V$,
$R$ is a meridian disk of $V$.
 Then a small isotopy moves $R$
so that $R \cap P=\bdd R \subset \bdd P$.

 Suppose first that
$\bdd R$ and a component of $\bdd A$
cobounds an annulus, say $A'$, in $\bdd V - \bdd t$. 
 Then $R \cup A'$ gives a disk properly embedded in $V$
and intersecting $t$ in a single point
when it is pushed into int\,$V$ with its boundary fixed.
 Note that the interior of this disk
intersects $P$ in at most one loop,
since we can recover $P$ from $R \cup A$ by tubing along a subarc of $t$.

 If such an annulus $A'$ does not exist,
then the loops $\bdd (R \cup A)$ divide $\bdd V$
into three annuli,
two of them intersect $\bdd R$, 
and each of these two annuli contains a point of $\bdd t$.
 Hence the other annulus and $A$
form a torus intersecting $t$ in a single point, 
a contradiction.
\end{proof}

\begin{lemma}\label{lem:BoundaryCompressible}
Let $P$ be a disk with two or three holes
properly embedded in $V$
so that $P$ is disjoint from $t$.
Suppose
that $P$ is $t$-incompressible and $t$-$\bdd$-compressible in $(V,t)$.
Let $P'$ be a surface obtained from $P$ by $t$-$\bdd$-compression.
If a component of $\bdd P'$ is $t$-inessential in $\bdd V$,
then one of the following occurs.
\begin{enumerate}
\renewcommand{\labelenumi}{(\theenumi)}
\item
 A component of $\bdd P$ bounds a disk $D$ in $V$
such that $D$ intersects $t$ transversely in a single point
and that $D \cap P = \bdd D$.
\item
 $P$ is meridionally compressible in $(V,t)$
and every component of $\bdd P$ is essential in $\bdd V$
$($ignoring the endpoints $\bdd t\/)$.
\end{enumerate}
\end{lemma}

\begin{proof}
 Let $Q$ be a $t$-$\bdd$-compressing disk of $P$,
and $P'$ the surface obtained
by $t$-$\bdd$-compressing $P$ along $Q$.
 Since $\bdd P'$ is $t$-inessential in $\bdd V$,
a component of $\bdd P'$ bounds in $\bdd V$ a disk
which is disjoint from $t$
or contains precisely one endpoint of $\bdd t$.
 In the former case,
$P'$ is $t$-compressible,
and hence $P$ is also $t$-compressible in $(V, t)$
before the $t$-$\bdd$-compressing operation.
 This contradicts our assumption.
 We consider the latter case.
 Then $P'$ is meridionally compressible,
and $P$ is also meridionally compressible in $(V, t)$.
 If $\bdd P$ is essential in $\bdd V$,
then we obtain the conclusion (2).
 Hence we may assume
that $\bdd P$ has a component
which is inessential in $\bdd V$
and bounds a disk $E$ in $\bdd V$
with (int\,$E \cap P = \emptyset)$.
 Since $P$ is $t$-incompressible,
$E$ contains one or two endpoints of $\bdd t$.
 In the former case,
we obtain the conclusion (1) immediately.
 We consider the latter case.
 In the course of $\bdd$-compression on $P$ along $Q$,
the boundary loop $\bdd E$ must be changed into two loops,
and each of them bounds a disk
which contains precisely one endpoint of $\bdd t$.
 The arc $\bdd Q \cap \bdd V$ has its two endpoints
in a single component of $\bdd P$.
 Hence $P'$ has an annulus component $A$.
 Then a component of $\bdd A$ bounds a disk $E'$ in $\bdd V$
such that $E'$ contains precisely one endpoint of $\bdd t$.
 The disk $A \cup E'$ is bounded by a loop of $\bdd P$,
and we can isotope the interior of this disk slightly off of $P$.
 Thus we obtain the conclusion (1).
\end{proof}

\begin{lemma}\label{lem:4-holes}
 Let $P$ be a disk with three holes properly embedded in $V$
so that $P$ is disjoint from $t$.
 If every component of $\bdd P$ is $t$-essential in $\bdd V$,
and if $P$ is $t$-incompressible and $t$-$\bdd$-compressible in $(V,t)$,
then one of the following occurs.
\begin{enumerate}
\renewcommand{\labelenumi}{(\theenumi)}
\item
 A component of $\bdd P$ bounds a disk $D$ in $V$
such that $D$ intersects $t$ transversely in a single point,
and that int\,$D$ is disjoint from $P$
or intersects $P$ in at most two loops.
\item
 There exists a $t$-$\bdd$-compressing disk $Q$ of $P$ in $(V,t)$
such that $\bdd Q \cap P$ connects distinct components of $\bdd P$
and that $\bdd P'$ is $t$-essential in $\bdd V$
where $P'$ is obtained from $P$ by $\bdd$-compression along $Q$.
\item
 $P$ is meridionally compressible in $(V,t)$
and every component of $\bdd P$ is essential in $\bdd V$
$($ignoring the endpoints $\bdd t\/)$.
\end{enumerate}
\end{lemma}

\begin{proof}
 Let $Q$ be a $t$-$\bdd$-compressing disk of $P$,
and $P'$ the surface obtained from $P$
by $\bdd$-compression on $P$ along $Q$.
 If a component of $\bdd P'$ is $t$-inessential in $\bdd V$,
we obtain the conclusion (1) or (3)
by Lemma \ref{lem:BoundaryCompressible}.
 Hence we may assume
that $\bdd P'$ is $t$-essential in $\bdd V$.
 Set $\alpha=\bdd Q\cap P$, and $\beta=\bdd Q\cap\bdd V$.
 If $\alpha$ connects distinct components of $\bdd P$,
we obtain the conclusion (2) in this lemma.
So we may assume
that $\alpha$ has both endpoints in the same component of $\bdd P$.
 Then $P'$ is
a union of an annulus $A$ and a disk with two holes $S$.

 By Lemma \ref{lem:pants}, either
(1) $S$ is $t$-compressible or meridionally compressible in $(V,t)$,
or
(2) there exists a $t$-$\bdd$-compressing disk $E$ of $S$ in $(V,t)$
such that the arc $\bdd E \cap S$ connects
two distinct components of $\bdd S$.

 Case (1).
 Since $S$ is $t$-compressible or meridionally compressible in $(V, t)$,
so is $A \cup S$.
 If $A \cup S$ is $t$-compressible, then $P$ is also $t$-compressible.
 This contradicts the assumption.
 Thus we may suppose
that $A \cup S$ is $t$-incompressible and meridionally compressible
in $(V,t)$.
 If the meridionally compressing disk is incident to $A$,
then every component of $\bdd A$ bounds a disk
which intersects $t$ in a single point.
 Since one component of $\bdd A$ is a component of $\bdd P$,
we obtain the conclusion (1).
 If the meridionally compressing disk is incident to $S$,
the conclusion (1) follows by Lemma \ref{lem:pants2}.

 Case (2).
 We can retake $E$
so that it is a $t$-$\bdd$-compressing disk of $A \cup S$
with the arc $\bdd E \cap \bdd V$ connecting 
distinct components of $\bdd A$ or $\bdd S$.
(Since $P$ is $t$-incompressible, so is $A$.
 Hence we can retake $E$
so that $E \cap A$ has no loop component.
 Further,
$E$ can be deformed
so that every arc of $A \cap E$ is essential in $A$
because the arc $\bdd E \cap \bdd V$ connects distinct components of $\bdd S$.
 When $E \cap A \ne \emptyset$,
we retake $E$ to be an outermost disk in original $E$
so that it is disjoint from $S$.)
 Let $\delta$ be a `dual' arc of $t$-$\bdd$-compression along $Q$.
 $P$ is recovered from $A \cup S$ by a band sum along $\delta$.
 See Figure \ref{fig:2-9}.

\begin{figure}[htbp]
\centering
\includegraphics[width=.7\textwidth]{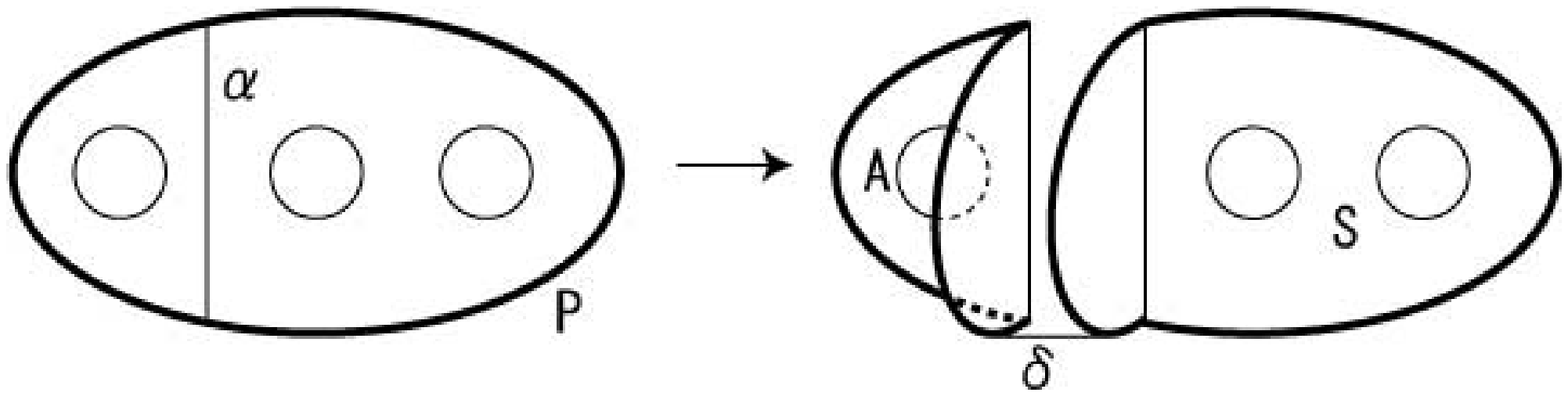}
\caption{}
\label{fig:2-9}
\end{figure}

 Suppose that $\delta \cap E = \emptyset$.
 Then $E$ forms a $t$-$\bdd$-compressing disk of $P$.
 Let $\tilde{P}$ be the surface obtained from $P$
by $t$-$\bdd$-compression along $E$.
 If $\bdd \tilde{P}$ is $t$-essential in $\bdd V$,
we obtain the conclusion (2).
 If not, we have the conclusion (1) or (3)
by Lemma \ref{lem:BoundaryCompressible}.

 Hence we can finally assume
that $\delta \cap E \neq \emptyset$.

\medskip

\noindent Claim.
 Among the closures of the components of $\bdd V - \bdd P'$,
let $U$ be the surface which contains the arc $\delta$.
 Then $U$ is a disk with two holes.

\begin{proof}
 By the definition of $\delta$,
one point of $\bdd \delta$ is in a component of $\bdd A$
and the other is in $\bdd S$.
 On the other hand,
the arc $\bdd E \cap \bdd V$ connects
either two boundary components of $A$
or those of $S$.
 Hence $U$ has at least three boundary components.
 Recall that the loops $\bdd P'$ are $t$-essential in $\bdd V$
as shown at the first step of this proof of lemma.
 Hence $U$ is a disk with two holes.
\end{proof}

 By this claim, we can move $\delta$ by an isotopy
so that $\delta \cap E=\emptyset$.
 Then this case comes to the previous case.
\end{proof}

\section{Surfaces in a handlebody of genus two with a core}

 We recall a lemma on  `essential'
surfaces
in a pair of a handlebody of genus two
and a core in it.

\begin{lemma}[Lemma 3.10 in \cite{GHY}]\label{lem:SurfaceIn20}  
Let $W$ be a handlebody of genus two,
and $K$ a core loop in $W$.
 Let $F$ be a compact orientable 2-manifold properly embedded in $W$
so that $F$ is transverse to $K$.
 Suppose that $F$ is $K$-incompressible and $K$-$\bdd$-incompressible.
 Then $F$ is a disjoint union of finitely many surfaces as below:
\begin{enumerate}
\renewcommand{\labelenumi}{(\theenumi)}
\item
a 2-sphere disjoint from $K$;
\item
a 2-sphere which bounds a trivial 1-string tangle in $(W,K)$;
\item
an essential disk of $W$ disjoint from $K$;
\item
an essential disk of $W$ intersecting $K$ transversely in a single point; 
\item
a torus bounding a solid torus
which forms a regular neighborhood of $K$ in $W$.
\end{enumerate}
\end{lemma}



\section{Weakly $K$-reducible splittings}

 We recall some results on $K$-reducible and weakly $K$-reducible splittings
in this section.

\begin{definition}
A $(1,1)$-splitting $(M,K)=(V_1,t_1)\cup_{H_1} (V_2,t_2)$
is called {\it $K$-reducible} if there are $K$-compressing disks
$D_1$ and $D_2$ of $H_{1}$ in $V_1$ and $V_2$ respectively
such that $\bdd D_1\cap\bdd D_2=\emptyset$.
\end{definition}

\begin{definition}
A knot $K$ in $M$ is called a {\it core knot}
if its exterior is a solid torus.
\end{definition}

Note that a knot in the $3$-sphere is a core knot
if and only if it is the trivial knot.


\begin{proposition}[Lemma 3.2 in \cite{Hy3}]\label{lem:11weak}
 Let $M$ be the $3$-sphere or a lens space other than $S^2 \times S^1$,
and $K$ a knot in $M$.
 Let $(M,K)=(V_1,t_1)\cup_{H_1} (V_2,t_2)$ be a $(1,1)$-splitting.
 If it is weakly $K$-reducible,
then one of the following occurs:
\begin{enumerate}
\renewcommand{\labelenumi}{(\theenumi)}
\item
 $K$ is a trivial knot;
\item
 $K$ is a core knot in a lens space;
\item
 $K$ is a 2-bridge knot in the $3$-sphere;
\item
 $K$ is a connected sum of a core knot in a lens space
and 2-bridge knot in the $3$-sphere.
\end{enumerate}
 When the $(1,1)$-splitting $H_1$ is $K$-reducible,
$K$ is trivial.
\end{proposition}

\begin{definition}
A $(2,0)$-splitting $(M,K)=(W_1,K)\cup_{H_2} (W_2,\emptyset)$
is called {\it $K$-reducible} if there are a $K$-compressing
disk $D_1$ of $H_2$ in $(W_1,K)$ and an essential disk $D_2$
in $W_2$ such that $\bdd D_1=\bdd D_2$ in $H_2$.
\end{definition}

\begin{definition}
A knot $K$ in $M$ is called a {\it split knot}
if its exterior cl$(M-N(K))$ is reducible. 
A knot $K$ is called {\it composite\/}
if there is a $2$-sphere $S$ embedded in $M$
such that $S$ is separating in $M$,
that $S$ intersects $K$ transversely in precisely two points
and that the annulus $S \cap E(K)$ is
incompressible and $\bdd$-incompressible in $E(K)$.
\end{definition}

\begin{proposition}[Proposition 2.9 in \cite{GHY}]\label{prop:KReducible}
A $(2,0)$-splitting $(M, K)=(W_1, K) \cup_{H_2} (W_2, \emptyset)$
is $K$-reducible
if and only if $K$ is either a core knot or a split knot.
\end{proposition}

\begin{proposition}[Proposition 2.14 in \cite{GHY}]\label{prop:20weak}
 A $(2,0)$-splitting $(M, K)=(W_1, K) \cup_{H_2} (W_2, \emptyset)$
is weakly $K$-reducible
if and only if one of the following occurs:
\begin{enumerate}
\renewcommand{\labelenumi}{(\theenumi)}
\item
 the $(2,0)$-splitting $H_2$ is $K$-reducible;
\item
 the $(2,0)$-splitting $H_2$ is meridionally stabilized; or
\item
$K$ is a composite knot.
\end{enumerate}
\end{proposition}

\begin{lemma}\label{lem:satellitediagram}
If a $(1,1)$-splitting $H_1$ admits a satellite diagram,
one of the following holds: 
\begin{enumerate}
\renewcommand{\labelenumi}{(\theenumi)}
\item
the knot $K$ is the trivial knot;
\item
the knot $K$ is the torus knot;
\item
the knot $K$ is a satellite knot;
\item 
the $(1,1)$-splitting $H_1$ admits a satellite diagram of a longitudinal slope. 
\end{enumerate}
\end{lemma}

This lemma is the correction and detailed account 
of the two sentences right before Theorem 1.2 in \cite{Hy3}.

\begin{proof}
 There are an essential loop $\ell$ in $H_1$
and a canceling disk $C_i$ of $t_i$ in $(V_i, t_i)$
such that the arc $\gamma_i=\bdd C_i \cap H_1$ is disjoint from $\ell$
for $i=1$ and $2$.

 If $\ell$ is meridional on $\bdd V_1$ or $\bdd V_2$, say $\bdd V_1$,
then $\ell$ bounds a meridian disk $R$ of $V_1$
such that $R$ is disjoint from $C_1$.
Doing surgery $H_1$ along $R$,
we obtain a $2$-sphere
on which $K$ has a $1$-bridge diagram.
 This implies that $K$ is trivial.

 Hence, it is enough 
to show that the conclusion (1), (2) or (3) holds
when $\ell$ is non-meridional and non-longitudinal 
both on $\bdd V_1$ and on $\bdd V_2$.
 Let $N(\ell)$ be a very thin neighborhood of $\ell$ in $H_1$,
and $A=$cl\,$(H_1 - N(\ell))$ the complementary annulus.
 Then a regular neighborhood $X$ of $A \cup C_1 \cup C_2$ is a solid torus,
and we denote its boundary torus by $T$.
 $T$ is incompressible in cl\,$(M-X)$.
 If $T$ has neither a $K$-compressing disk 
nor a meridionally compressing disk in $(X, K)$,
then $K$ is a satellite knot.
 This is the conclusion (3).

 Hence we can assume that $T$ has 
a $K$-compressing or meridionally compressing disk $D$ in $(X, K)$.
 We can take $D$ 
so that $\bdd D$ intersects each component of $\bdd A$ 
transversely in a single point.
 We prove the next claim in the last three paragraphs in this proof.
 Recall that $\gamma_i = \bdd C_i \cap A$ is an arc for $i=1$ and $2$,
and $\gamma_1 \cup \gamma_2$ forms a $1$-bridge diagram of $K$ in $A$.

\noindent
{\bf Claim}.
{\it $K$ is the trivial knot,
or we can retake $D$, $C_1$ and $C_2$
so that $D$ intersects $A$ transversely in a single arc $\delta$
such that (A) $\delta$ is disjoint from $\gamma_1 \cup \gamma_2$
or (B) $\delta$ intersects $\gamma_1 \cup \gamma_2$
transversely in a single point.}

 In Case (A) of Claim, 
$K$ has a $1$-bridge diagram on the disk $A - \delta$,
and hence is the trivial knot,
which is the conclusion (1).
 In Case (B) of Claim,
we can assume, without loss of generality,
that $\delta$ intersects $\gamma_1$ in a single point.
 Then we can move $C_2$ by isotopy near $\gamma_2 \subset A-\ell$
so that $\gamma_2$ intersects $\gamma_1$
precisely at its endpoints $\bdd \gamma_2 = K \cap A$.
 This implies that $K$ is a torus knot,
and we obtain the conclusion (2).

 Now, we prove the claim.
 Let $X_i$ be the half solid torus $X \cap V_i$ for $i=1$ and $2$.
 We can take $D$
so that $D$ intersects $X_1$ in a meridian disk $D_1$ 
with $D_1 \cap A = \bdd D_1 \cap A = \delta$ an arc
and possibly peripheral disks $Q_i$ satisfying the following conditions:
(i) $\bdd Q_i$ is in $A$; 
and (ii) $Q_i$ cuts a ball containing $t_1$ from $X_1$.
 Moreover, we take $D$ 
so that the number $n_1$ of peripheral disk components 
of $D \cap X_1$ is minimal.
 Then the surface $D \cap X_2$ is $t_2$-incompressible in $(X_2, t_2)$.
 $D_1$ intersects $t_1$ transversely in at most one point.
 When $D \cap X_1$ does not have a peripheral disk,
$D_i = D \cap X_i$ is a disk for $i=1$ and $2$
such that $D_1 \cap D_2 = \delta$.
 As in the next paragraph, we can take $C_i$
so that it is disjoint from $D_i$ if $D_i$ is disjoint from $t_i$,
and intersects $D_i$ in a single arc connecting $t_i$ and $A$
if $D_i$ intersects $t_i$ in a single point.
 Then Claim follows.
 Note that $D \cap X_1$ has no peripheral disk components
if $D_1$ does intersect $t_1$ in a single point.

\begin{figure}[htbp]
\centering
\includegraphics[width=.4\textwidth]{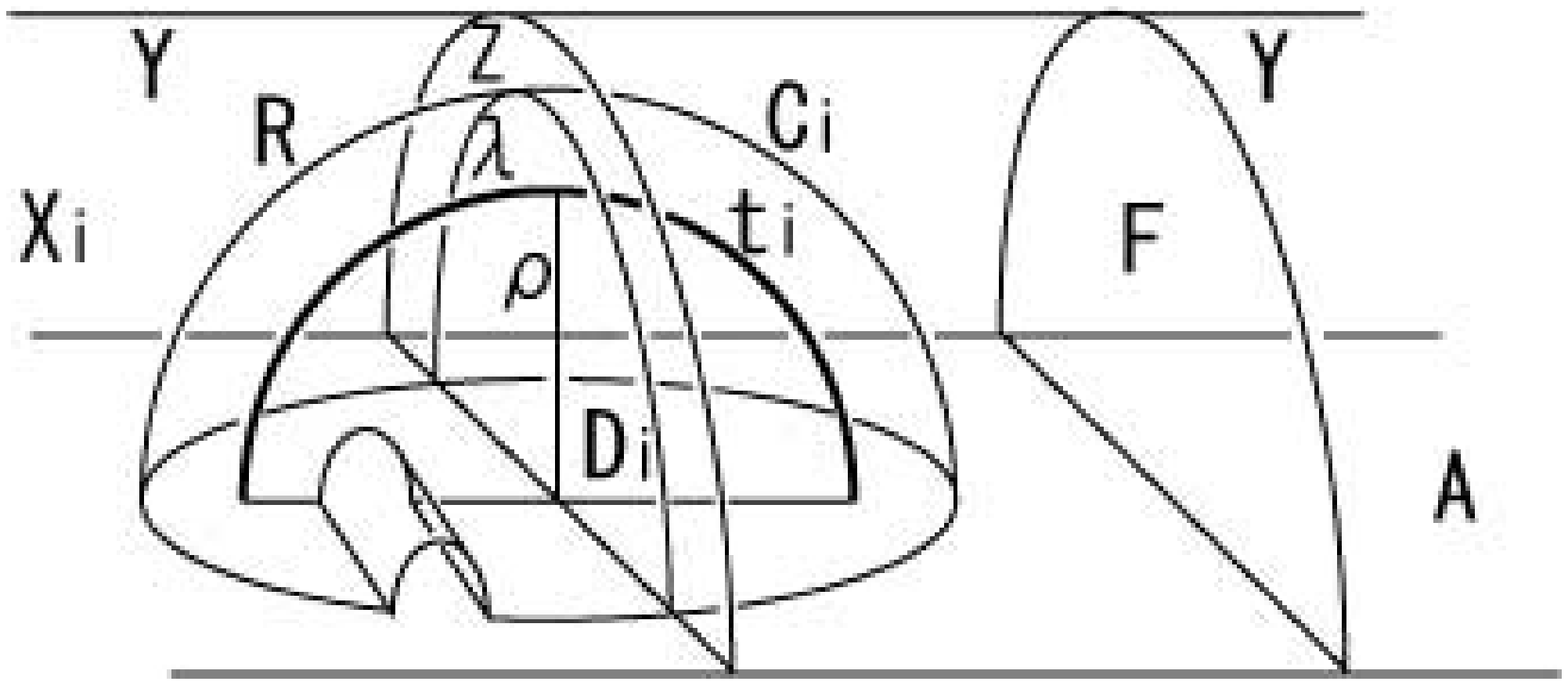}
\caption{}
\label{fig:satellitediagram2}
\end{figure}

 By a standard innermost circle argument,
we can take $C_i$ so that it intersects $D_i$ only in arcs.
 Note that $D_i$ contains an intersection arc, say $\rho$, 
connecting the point $D_i \cap t_i$ and the arc $\bdd D_i \cap A$
when $D_i$ intersects $t_i$ in a single point.
 By a standard outermost arc argument,
we can retake $C_i$ 
so that it is disjoint from $D_i$ if $D_i \cap t_i = \emptyset$,
and it intersects $D_i$ in $\rho$ and a parallel family of arcs
separating $\rho$ and the arc $\bdd D_i \cap \bdd X_i$
if $D_i$ intersects $t_i$ in a single point.
 In the former case, we are done.
 In the latter case,
if the family of arcs is empty, then we have obtained the desired situation.
 Hence we can assume that the family is non-empty.
 Let $E$ be an outermost disk
cut off from $C_i$ by an outermost intersection arc, 
say $\eta$, of $C_i \cap D_i$.
 We can take $E$ so that it is disjoint from $t_i$.
 $\eta$ divides $D_i$ into two disks, 
one of which, say $P$, is disjoint from $\rho$.
 Then the disk $F = E \cup P$ is a meridian disk of $X_i$
with $\bdd F$ intersecting the annulus $\bdd X_i \cap \bdd X$
in a single essential arc.
 We isotope $F$ slightly off of $D_i$ and $E$.
 Note that $F$ is disjoint from $t_i$.
 A standard outermost arc argument on $F$
allows us to deform $C_i$ to be disjoint from $F$.
 Let $R$ be the closure of a component of $C_i - D_i \cap C_i$ 
with $t_i \subset R$.
 We discard the other components.
 Let $\lambda$ be an arc of $(\bdd R) \cap D_i$ 
which is the outermost on $D_i$ among the arcs $(\bdd R) \cap D_i$.
 We can choose $\lambda$
so that its outermost disk, say $Z$, is disjoint from $\rho$.
 See Figure \ref{fig:satellitediagram2}.
 Let $Y$ be one of the disks
obtained by cutting the annulus $\bdd X_i \cap \bdd X$
along the arcs $\bdd D_i \cap \bdd X_i$ and $\bdd F \cap \bdd X_i$.
 We add the disk $Z \cup Y \cup F$ to $R$, 
and call the resulting disk $R$ again.
 If we choose $Y$ adequately,
then we can isotope $R$ near the disk $Z \cup Y \cup F$ off of $Z \cup Y \cup F$
so that the number of arcs $R \cap D_i$ is decreased.
 Repeating operations as above,
we obtain a canceling disk $C_i$ of $t_i$
which intersects $D_i$ only in the arc $\rho$.
 This completes the proof of Claim
in the case where $D \cap X_1$ has no peripheral components.

\begin{figure}[htbp]
\centering
\includegraphics[width=.8\textwidth]{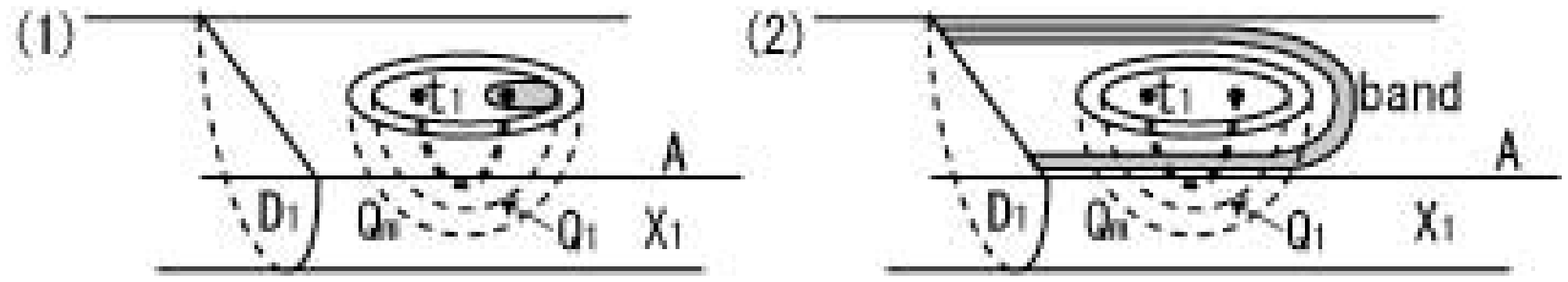}
\caption{}
\label{fig:SatelliteDiagram}
\end{figure}

 Thus we can assume 
that $D_1$ is disjoint from $t_1$
and $D \cap X_1$ has one or more peripheral disk components,
say $Q_1, Q_2, \ldots, Q_m$,
where $Q_1$ is the innermost one.
 Let $Q'_1$ be the disk in $A$ bounded by $\bdd Q_1$.
 We take $C_2$ 
so that it intersects the $t_2$-incompressible surface $D \cap X_2$
transversely in arcs and zero loops.
 Let $n_2$ denote the number of intersection arcs of $C_2 \cap (D \cap X_2)$.
 We retake $D$ and $C_2$
so that the pair $(n_1, n_2)$ is minimal lexicographically.
 If $n_2=0$,
then $K$ has a $1$-bridge diagram in $Q'_1$,
and hence is the trivial knot.
 Hence we can assume $n_2 > 0$.
 Let $e$ be one of the outermost intersection arcs in $C_2$.
 We isotope $D$ near $e$ along the outermost disk.
 If $e$ has an endpoint in $t_2$,
then $Q_1$ is deformed into a peripheral disk
intersecting $t_1$ transversely in a single point,
we can isotope it into $X_2$
to reduce the number of peripheral disks of $D \cap X_1$
(Figure \ref{fig:SatelliteDiagram} (1)).
 This is a contradiction.
 Hence we can assume that $e$ has both endpoints in $A$.
 Then, after the isotopy, a band is attached to the surfaces $D \cap X_1$.
 If the band connects $D_1$ and $Q_m$,
then they are deformed into a meridian disk of $X_1$,
and the number of peripheral disks is reduced.
 If the band connects two peripheral disks,
then they are deformed into a single peripheral disk,
which can be isotoped out of $X_1$ into $X_2$ 
without intersecting $K$.
 This operation reduces the number of peripheral disks.
 If the band is attached to $Q_1$,
then it is deformed into an annulus
which has a meridionally compressing disk $G$ in $(X_1, t_1)$.
 Then $D$ is a meridionally compressing disk of $T$
rather than a $K$-compressing disk of $T$.
Doing surgery $D$ along the meridionally compressing disk $G$,
we obtain a new meridionally compressing disk of $T$
and a $2$-sphere intersecting $K$ in two points.
 We discard this $2$-sphere.
 Then $Q_1$ is deformed 
into a peripheral disk intersecting $t_1$ in a single point.
 We can isotope it into $X_2$ to decrease the number of peripheral disks.
 Thus we can assume 
that the band is attached to $D_1$.
 Then it is deformed into an annulus
one of whose boundary loop is parallel to $\bdd Q_m$ in $A$.
 See Figure \ref{fig:SatelliteDiagram} (2).
 Hence the annulus has a $t_1$-compressing disk in $X_1$.
 We do surgery on $D$ along the $t_1$-compressing disk.
 Then the annulus is deformed 
into a disjoint union of a peripheral disk and a disk isotopic to $D_1$,
and $D$ is deformed into a disk and a $2$-sphere.
 We discard this $2$-sphere, which contains the new peripheral disk.
 This operation does not change the number of peripheral disks,
but decreases the number of intersection arcs of $C_2$ and $D \cap X_2$,
which is a contradiction.
\end{proof}

Since the trivial knot is either a core knot or split knot, 
we have the next proposition by 
Proposition \ref{prop:KReducible} and Lemma \ref{lem:satellitediagram}.

\begin{proposition}\label{prop:satellitediagram}
Suppose $(M,K)$ has a $(1,1)$-splitting $H_1$ and $(2,0)$-splitting $H_2$. 
Further, we suppose that $H_1$ admits a satellite diagram. Then 
one of the following holds: 
\begin{enumerate}
\renewcommand{\labelenumi}{(\theenumi)}
\item
the $(2,0)$-splitting $H_2$ is $K$-reducible; 
\item
the knot $K$ is the torus knot;
\item
the knot $K$ is a satellite knot;
\item 
the $(1,1)$-splitting $H_1$ admits a satellite diagram of a longitudinal slope. 
\end{enumerate}
\end{proposition}

 Further, we will use the next proposition by T. Kobayashi.

\begin{proposition}[Proposition 3.4 in \cite{K2}]\label{thm:koba}
 Let $M$ be a closed orientable 3-manifold, and $L$ a link in $M$.
 Assume that $M$ has a 2-fold branched cover
with branch set $L$.
 Let $H_i$ be $g_i$-genus $n_i$-bridge splitting of $(M,L)$
for $i=1$ and $i=2$,
and $W$ a genus $g_2$ handlebody bounded by $H_2$ in $M$.
 Suppose that $H_1$ is contained in the interior of $W$,
and that there is
an $L$-compressing or meridionally compressing disk $D$ of $H_2$
in $(W, L \cap W)$ with $D \cap H_1 = \emptyset$.
 Then either
$(i)$ $M = S^3$ and $L = \emptyset$ or $L$ is the trivial knot,
or $(ii)$ the splitting $H_2$ is weakly $L$-reducible.
\end{proposition}



\section{General case}

 Recall that $M$ is the 3-sphere or a lens space ($\neq S^2\times S^1$)
and $K$ is a knot in $M$.
 Let $H_1$ be a torus giving a 1-genus 1-bridge splitting
$(M,K)=(V_1,t_1)\cup_{H_1}(V_2,t_2)$,
and $H_2$ a genus 2 surface giving a $(2,0)$-splitting
$(M, K) = (W_1, K) \cup_{H_2} (W_2, \emptyset)$.
 We assume that $H_1$ and $H_2$ intersect transversely
in non-empty collection of finitely many loops
each of which is $K$-essential both in $H_1$ and in $H_2$.
 If a loop $l$ of $H_1\cap H_2$ is inessential in $H_1$
(ignoring the intersection points $K \cap H_1$),
then by the definition of $K$-essentiality
$l$ bounds a disk $D$ in $H_1$
such that $D$ intersects $K$ in two points.
 The torus $H_1$ contains
zero or even number of essential loops of $H_1 \cap H_2$
since $H_2$ is separating in $M$.

 We consider in this section the $\lq\lq$general case"
where $H_1 \cap H_2$ contains a parallel family of three loops
in $H_2$.
 If $H_1 \cap H_2$ consists of seven or more loops,
then $H_1 \cap H_2$ contains such a family
because the surface $H_2$ of genus two contains
at most three essential loops
which are pairwise non-parallel and disjoint.

 We can apply
precisely the same argument as in Section 4 in \cite{Hy3}
to this case.
 There
two $1$-genus $1$-bridge splittings $H_1$ and $H_2$ of $(M, K)$
are considered,
and $H_1 \cap H_2$ contains a family of three loops on $H_2 - K$.
 Then we obtain the next proposition
which is similar to Proposition 4.1 in \cite{Hy3} 
via Proposition \ref{prop:satellitediagram}.
 We omit the proof.

\begin{proposition}\label{prop:general}
 Suppose that the intersection $H_1 \cap H_2$ contains
a parallel family of three loops in $H_2$.
then at least one of the conditions $(a)\sim(e)$ of 
Theorem \ref{thm:general} holds. 

\end{proposition}


 We will consider in the following sections the case
where $H_1 \cap H_2$ does not contain
such a parallel family of three loops,
and $H_1 \cap H_2$ is composed of six or less number of loops.

\section{When $|H_1\cap H_2|=3$}

 We consider in this section the case
where $H_1$ and $H_2$ intersect each other
in three loops each of which is $K$-essential both in $H_1$ and in $H_2$.
%

In $H_1$,
the intersection loops $H_1 \cap H_2$ consists of
either (I) two parallel essential loops and a single inessential loop
or (II) three parallel inessential loops.
In both cases,
every inessential loop is $K$-essential in $H_1$,
and bounds a disk intersecting
$K$ in two points.
By Proposition \ref{prop:general} we may assume that
in $H_2$,
the intersection $H_1 \cap H_2$ consists of
either (A) three non-separating loops no pair of which is parallel,
or (B)
an essential separating loop
and two parallel essential non-separating loops.
In Case (A),
the three intersection loops divide
$H_2$
into two disks with two holes.
In Case (B),
they divide
$H_2$
into an annulus, a disk with two holes and a torus with one hole.
See Figure \ref{fig:l=3}.

\begin{figure}[htbp]
\centering
\includegraphics[width=.7\textwidth]{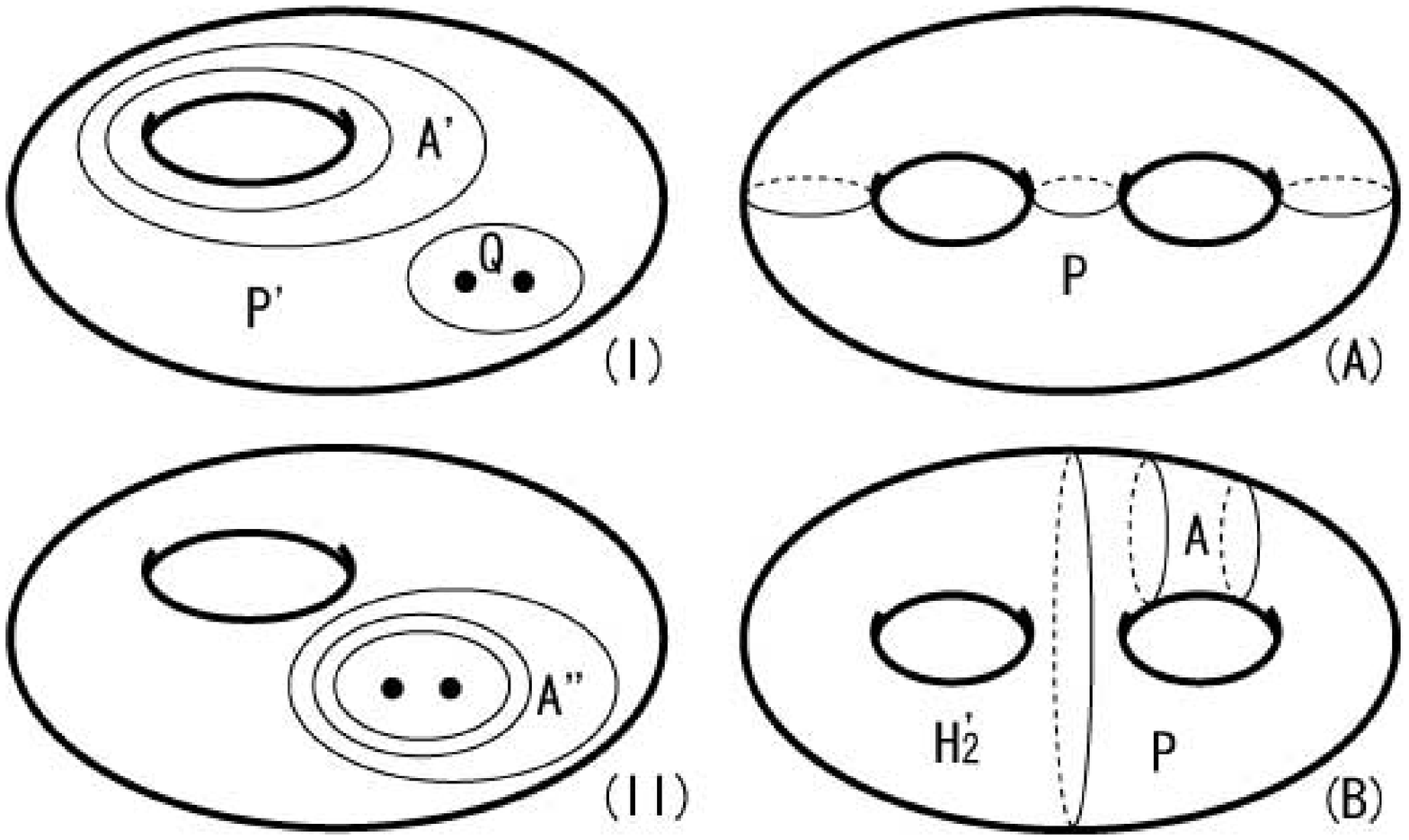}
\caption{}
\label{fig:l=3}
\end{figure}

In order to prove Theorem \ref{thm:general} (2),
we prepare the next lemma.

\begin{lemma}\label{lem:Pcompressible}
If, for $i=1$ or $2$,
the intersection $H_2 \cap V_i$ contains a component $P$
that is a disk with two holes,
then either
$(i)$ we can isotope $H_2$ in $(M, K)$
so that $H_1$ and $H_2$ intersect each other
in smaller number of loops
each of which is $K$-essential both in $H_1$ and in $H_2$,
or
$(ii)$ $P$ is $t_i$-compressible or meridionally compressible in $(V_i,
t_i)$.
Hence we can assume that $P$ satisfies $(ii)$
to prove Theorem \ref{thm:general} $(2)$.
\end{lemma}

\begin{proof}
$H_2 \cap V_i$ consists of the single component $P$
in both cases (A) and (B).
 By Lemma \ref{lem:pants},
either
(i)$'$ $P$ has a $t_i$-$\bdd$-compressing disk $D$ in $(V_i, t_i)$
such that the arc $\bdd D \cap P$
connects two distinct components of $\bdd P$
and
that the boundary loops of the annulus
obtained by $t_1$-$\bdd$-compressing along $D$
are $K$-essential in
$H_1$,
or
(ii)$'$ $P$ is $t_i$-compressible or meridionally compressible in
$(V_i,t_i)$.
 The latter case (ii)$'$ is precisely the conclusion (ii) of this lemma.
 In the former case (i)$'$,
we obtain the conclusion (i)
by isotoping
$H_2$ along
$D$.
Note that after this isotopy
the intersection loops $H_1 \cap H_2$ are $K$-essential in $H_1$.
The intersection $H_2 \cap V_i$ is deformed into an annulus,
and a band is attached to 
$H_2 \cap V_j$,
where $\{ i,j \} = \{ 1,2 \}$.
Hence the intersection loops $H_1 \cap H_2$
are also $K$-essential in $H_2$ after this isotopy.
\end{proof}


\medskip

\noindent{\bf Proof of Theorem \ref{thm:general} (2).}

{\bf Case A.}
Let $P_i$ be the disk with two holes $H_2 \cap V_i$ for $i=1$ and $2$.
 By Lemma \ref{lem:Pcompressible},
we
may assume
that $P_i$ is $t_i$-compressible or meridionally compressible
in $(V_i,t_i)$
for $i=1$ and $2$.
 We perform $t_i$-compression or meridionally compression
on $P_i$ for $i=1$ and $2$,
to obtain a $t_i$-compressing or meridionally compressing disk
of $H_1$ on both sides of $H_1$.
This shows that
$H_1$ is weakly $K$-reducible
since these disks are bounded by the loops of $H_1 \cap H_2$.


{\bf Case B.}
We obtain from $H_2$ an annulus $A$,
a disk with two holes $P$
and a torus with one hole $H'_2$
by cutting along the three loops $H_1 \cap H_2$.
 We may assume without loss of generality
that $P$ is in $V_1$.
Then $A$ and $H_2'$ are in $V_2$.
 By Lemma \ref{lem:Pcompressible},
we can assume that $P$ is $t_1$-compressible or meridionally compressible
in $(V_1, t_1)$,
and compressing operation on a copy of $P$ yields a disk $D_1$
which is bounded by a loop of $\bdd P = H_1 \cap H_2$
and intersects the arc $t_1$ transversely in at most one point.
By Lemma \ref{lem:H},
$A$ is $t_2$-compressible or $t_2$-$\bdd$-compressible
in $(V_2, t_2)$.

 In the former case,
we perform a $t_2$-compressing operation on $A$,
to obtain a disk $D_2$
which is bounded by a loop of $H_1 \cap H_2$
and is disjoint from the arc $t_2$.
 Then the disks $D_1$ and $D_2$ show
that
$H_1$ is weakly $K$-reducible.
 This is the conclusion (b) of this theorem.
Thus we may assume that $A$ has a $t_2$-$\bdd$-compressing disk $D$
in $(V_2, t_2)$.

 First, we consider Case (I).
 Precisely one of the three loops $H_1 \cap H_2$
is inessential in $H_1$.
 If one component of $\bdd A$ is inessential in $H_1$,
then the annulus $A$ is $t_2$-compressible by Lemma \ref{lem:annulus}.
 Then, by compressing $P$ and $A$, we can again see
that $H_1$ is weakly $K$-reducible.
 Therefore,
we may suppose that every component of $\bdd A$ is essential in $H_1$.
 Let $A'$, $P'$ and $Q$ 
be the annulus, the disk with two holes and the disk
obtained by cutting $H_1$ along $H_1 \cap H_2$.
$Q$ intersects $K$ in two points.
 If $A$ is parallel to to $A'$ in $(V_2, t_2)$, 
then we have the conclusion (a).
Suppose not.
By $t_2$-$\bdd$-compressing $A$ along $D$, 
we obtain a disk $G$ 
such that $G$ is disjoint from $t_2$, 
$\bdd G \cap \bdd A = \emptyset$ and $\bdd G$ is inessential in $H_1$.
Since $A$ is not $\bdd$-parallel,
$\bdd G$ is $t_2$-essential in $H_1$,
and hence $\bdd G \subset H_1 - A'$.
 If $\bdd D_1$ is essential in $H_1$, 
then $G$ and $D_1$ show $H_1$ is weakly $K$-reducible.
 This is the conclusion (b).
 We may assume that $\bdd D_1$ is inessential in $H_1$.
 Then $\bdd D_1 = \bdd Q$ and 
$D_1$ is parallel to $Q$ ignoring $t_1$.
 Since $Q$ intersects $K$ in two points,
$D_1$ is disjoint from $t_1$.
Thus $D_1$ and $G$ are disjoint from $\bdd A$,
and show that $H_1$ has a satellite diagram.
By Proposition \ref{prop:satellitediagram}, 
we have the conclusion $(c)\sim (f)$.


We consider Case (II).
The three loops $H_1 \cap H_2$ are inessential in $H_1$,
and parallel in $H_1 - K$.
Recall that $D$ is a $t_2$-$\bdd$-compressing disk of $A$
in $(V_2,t_2)$.
 Since the arc $\bdd D \cap H_1$
connects distinct components of $\bdd A$,
it is in an annulus component, say $A''$, of $H_1 -  \bdd A$.
Note that $A''$ is disjoint from the two endpoints $\bdd t_2$.
Hence the annulus $A$ is $t_2$-parallel to $A''$ in $(V_2, t_2)$.
If int$A''$ does not contain the boundary loop $\bdd H'_2$,
then we can isotope $H_2$ along the parallelism between $A$ and $A''$,
to reduce the number of intersection loops $H_1 \cap H_2$.
We obtain the conclusion $(a)$ of this theorem.
If int$A''$ contains the boundary loop $\bdd H'_2$,
then the solid torus of parallelism between $A$ and $A''$
entirely contains the torus with one hole $H'_2$
whose boundary loop is not null-homotopic in the solid torus.
 Thus we obtain a contradiction.

 This completes the proof of Theorem \ref{thm:general} (2).
$\qed$


\section{$|H_1\cap H_2|= 4(A), 5$ or $6$}

 We consider in this section
the cases $|H_1\cap H_2|= 5, 6$
and the subcase (A) of the case $|H_1\cap H_2|= 4$ as below
simultaneously.
 See Figures \ref{fig:l=4}, \ref{fig:l=5} and \ref{fig:l=6}.
 In these cases,
the closures of the components of $H_2- (H_1\cap H_2)$
contains an annulus and a disk with two holes $P$.
 In the other subcase (B) of the case $|H_1 \cap H_2|=4$,
the four intersection loops
consists of two parallel pairs of loops
and each loop is non-separating in $H_2$.
 Recall that,
after Proposition \ref{prop:general},
we are under the assumption
that $H_1 \cap H_2$ does not contain
a parallel family of three loops in $H_2$.

\begin{figure}[htbp]
\centering
\includegraphics[width=.7\textwidth]{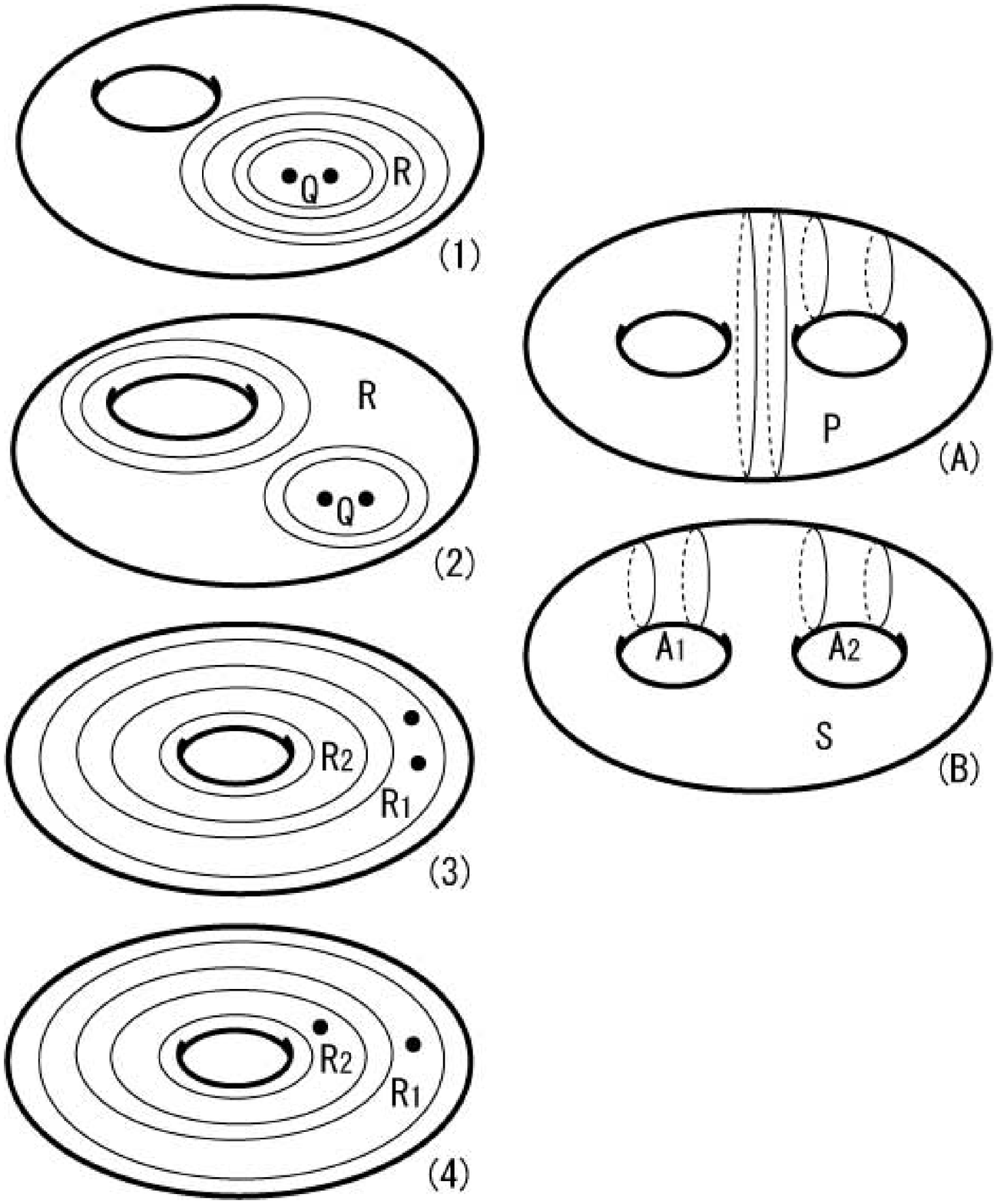}
\caption{}
\label{fig:l=4}
\end{figure}

\begin{figure}[htbp]
\centering
\includegraphics[width=.4\textwidth]{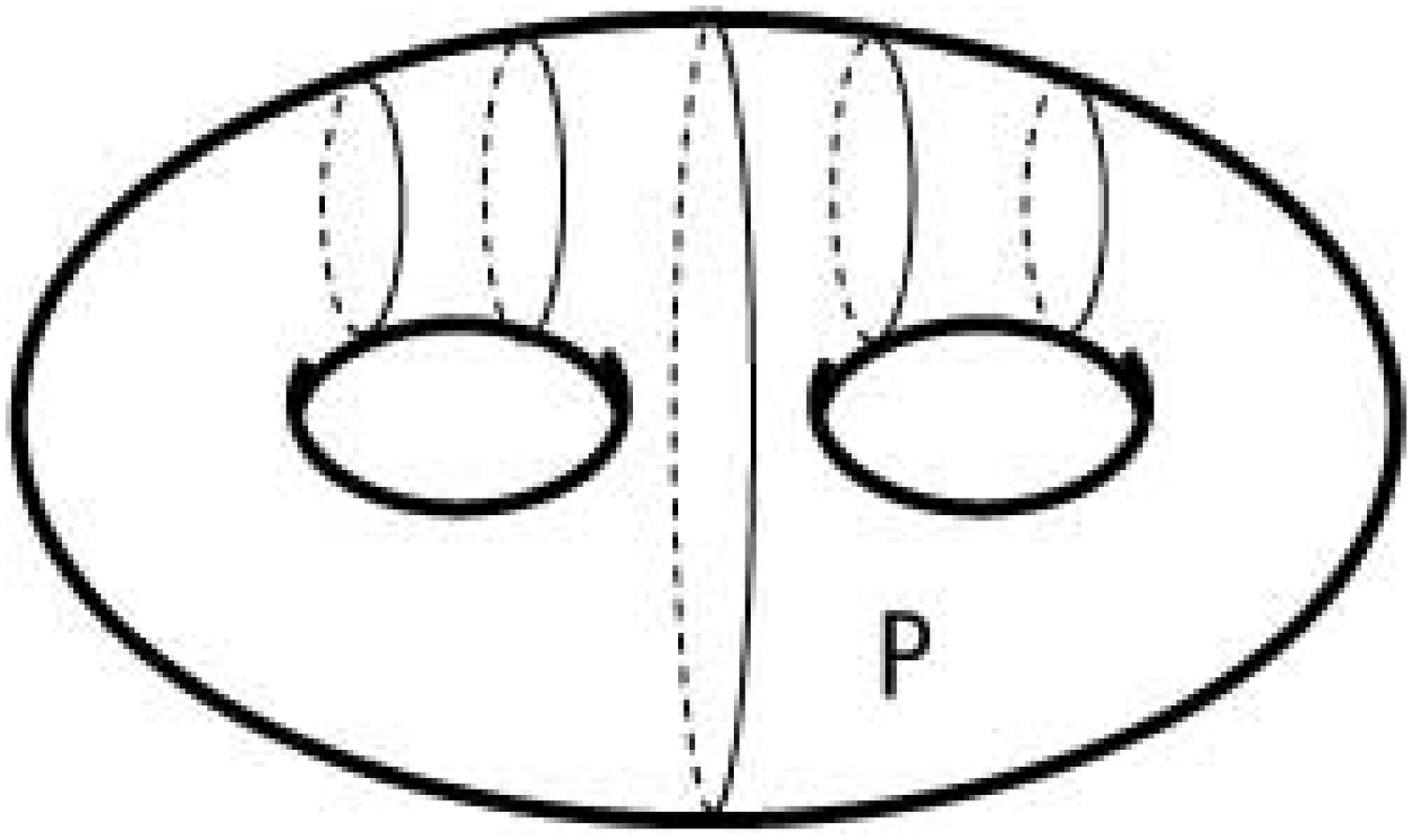}
\caption{}
\label{fig:l=5}
\end{figure}

\begin{figure}[htbp]
\centering
\includegraphics[width=.7\textwidth]{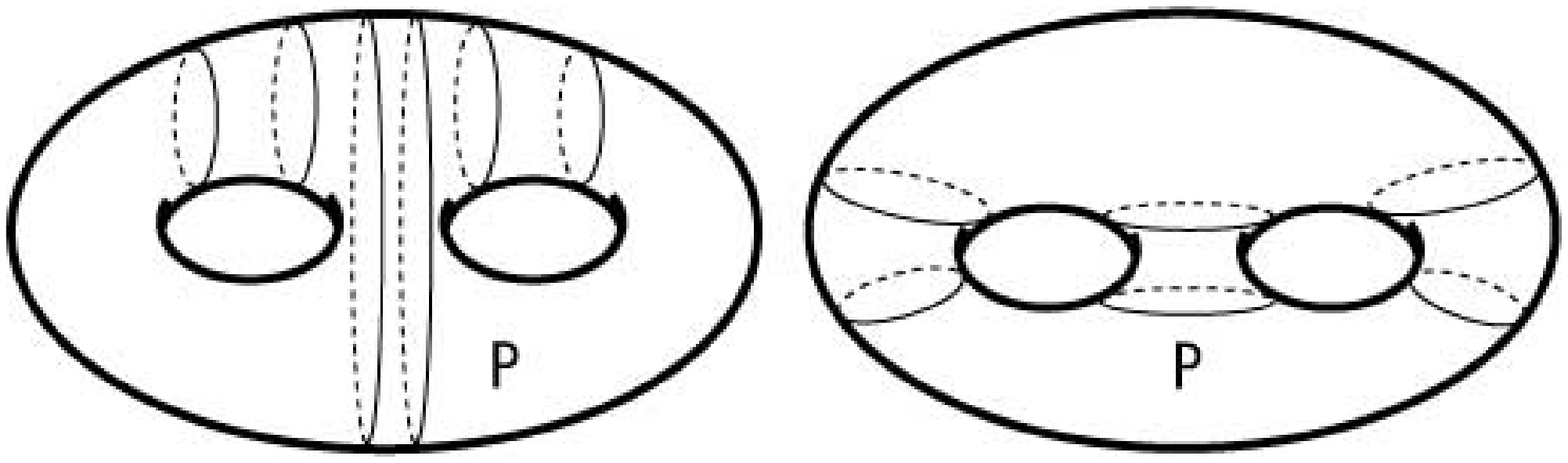}
\caption{}
\label{fig:l=6}
\end{figure}

Our goal in this section is the next proposition.

\begin{proposition}\label{prop:456}
 Under the above condition,
the conclusion $(a)$ or $(b)$ of Theorem \ref{thm:general} holds.
\end{proposition}

 To prove this proposition, we need the next lemma.
 We may assume without loss of generality
that the disk with two holes $P$ is contained
in the solid torus $V_1$ bounded by $H_1$.

\begin{lemma}\label{lem:456}
 One of the following occurs:
\begin{enumerate}
\renewcommand{\labelenumi}{(\theenumi)}
\item
a loop of $H_1\cap H_2$ bounds
in $(V_1,t_1)$
a $t_1$-compressing disk or a meridionally compressing disk
$($the interior of which may intersect $H_2)$; or
\item
the conclusion $(a)$ of Theorem \ref{thm:general} holds.

\end{enumerate}
\end{lemma}

\begin{proof}

 $P$ satisfies the conclusion (1) or (2) of Lemma \ref{lem:pants}
in $(V_1, t_1)$.
 In the former case, 
we obtain the conclusion (1) of this lemma. 
 In the latter case, let $D$ and $P'$ be the disk and the annulus 
in the conclusion (2) of Lemma \ref{lem:pants}.
 Let $\alpha = \bdd D \cap P$ and $\beta = \bdd D \cap \bdd V_1$.
 Let $F$ be one of the surfaces
obtained by cutting $H_1$ along the boundary loops $\bdd P$
such that $F$ contains the arc $\beta$.
 Since $\alpha$ connects two distinct components of $\bdd P$,
$F$ has at least two boundary components.
After $\bdd$-compression along $D$,
the boundary loops 
$\bdd P'$ are $t$-essential,
hence
$F$ is either a disk with two holes disjoint from $t_1$,
or an annulus which contains the two endpoints $\bdd t_1$.
Figure \ref{fig:456-2} indicates typical examples.
 If $F$ is an annulus containing the two endpoints $\bdd t_1$,
then each component of $\bdd P$ is essential in $H_1$.
 Hence a boundary component of the annulus 
$P'$ is essential,
and the other boundary component of 
$P'$ is inessential
in $\bdd V_1$.
 Lemma \ref{lem:annulus} shows that 
$P'$ is $t_1$-compressible.
 Then $P$ is also $t_1$-compressible in $(V_1, t_1)$
before the $\bdd$-compression along $D$.
 This is the conclusion (1).

\begin{figure}[htbp]
\centering
\includegraphics[width=.7\textwidth]{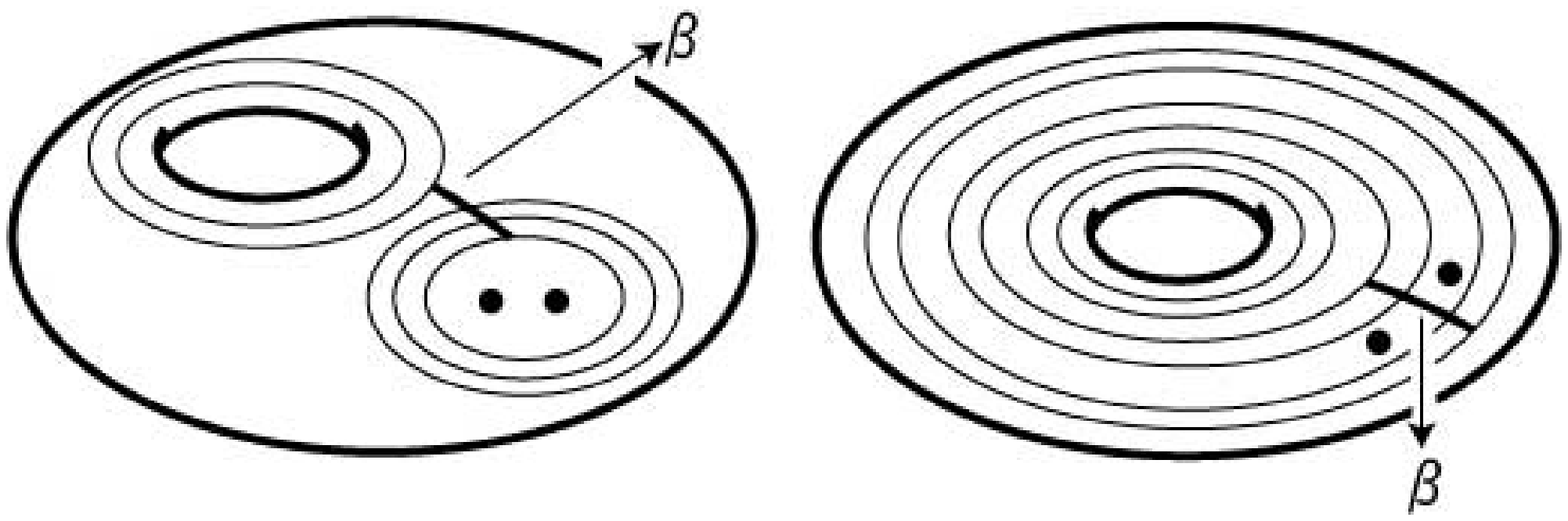}
\caption{}
\label{fig:456-2}
\end{figure}

Hence we may assume that $F$ is a disk with two holes.
Note that
$H_2$ may intersect $\text{int}\,F$.
By an adequate isotopy of
$D$ near
$\beta$,
we can move $D$ so that $\beta$ intersects
each component of $H_1\cap H_2$ in at most one point.
 Note that each component of $H_2 \cap V_1$ is
a torus with one hole, a disk with two holes or an annulus.
Therefore,
if $H_2 \cap V_1$ is $t_1$-compressible in $(V_1,t_1)$,
then we obtain the conclusion (1).
We may suppose
that $H_2\cap V_1$ is $t_1$-incompressible in $(V_1,t_1)$.
Then,
by a standard cut and paste argument,
we can retake the $t_1$-$\bdd$-compressing disk $D$ of $P$
so that each component of $D\cap(H_2\cap V_1)$ is an arc.
 Let $D'$ denote an outermost disk
which is cut off from $D$
by an outermost arc of $D\cap(H_2\cap V_1)$.
 Set $\beta'=\bdd D'\cap\bdd V_1$.
 The arc $\beta'$ connects distinct loops of $H_1\cap H_2$
since $\beta'\subset\beta$.
 Hence $D'$ is not incident
to a torus with one hole component of $H_2 \cap V_1$.

 Assume first that $\text{int}\,\beta'$ is contained
in a disk with two holes component of $F - H_2$.
 If $D'$ is incident to an annulus component of $H_2\cap V_1$,
we can obtain a disk $D''$ by $\bdd$-compression along $D'$
such that $D''$ is disjoint from
$t_1$.
 Note that $\bdd D''$ is parallel to a loop of $H_1 \cap H_2$
in $H_1 - \bdd t_1$.
 Thus we obtain the conclusion (1).
 If not,
$D'$ is incident to a disk with two holes component of $H_2 \cap V_1$.
 By isotoping $H_2$ along the disk $D'$,
we can decrease the number of intersection curves $H_1 \cap H_2$
by one.
 Thus we have the conclusion (2).

 Assume that $\text{int}\,\beta'$ is contained
in an annulus component $A$ of $F - H_2$.
 Note that $A$ is disjoint from the endpoints $\bdd t_1$
since $A \subset F$.
 If $D'$ is incident to an annulus component of $H_2 \cap V_1$,
then the annulus can be moved into $V_2$ by an isotopy.
 Thus we can cancel two intersection loops of $H_1\cap H_2$, 
and obtain the conclusion (2).
 If $D'$ is incident to a disk with two holes component $J$
of $H_2 \cap V_1$,
then we obtain a $t_1$-compressing disk of $J$
by a `$\bdd$-compression' on a copy of
$A$ along $D'$.
 This contradicts the assumption
that $H_2 \cap V_1$ is $t_1$-incompressible in $(V_1,t_1)$.

 This completes the proof of Lemma \ref{lem:456}.
\end{proof}

\noindent{\bf Proof of Proposition \ref{prop:456}.}
 If the conclusion (2) of Lemma \ref{lem:456} occurs,
then we are done.
 Suppose that the conclusion (1) of Lemma \ref{lem:456} occurs.
 That is, a loop of $H_1\cap H_2$ bounds a disk $E$ in $(V_1,t_1)$
such that 
$E$ intersects $t_1$ transversely in at most one point.

 If $H_2\cap V_2$ also contains a disk with two holes component
 (this is possible only in the case of $|H_1\cap H_2|=5$),
then by the same argument as in the proof of Lemma \ref{lem:456},
we obtain the conclusion that a loop of $H_1\cap H_2$ bounds a disk
$E'$ in $(V_2,t_2)$
such that 
$E'$ intersects $t_2$ transversely in at most one point.
 The disks $E$ and $E'$ show
that
$H_1$ is weakly $K$-reducible.
 This is the conclusion (1) of this proposition.

 If not,
$H_2\cap V_2$ consists of two annuli
(in Subcase (A) in the case of $|H_1\cap H_2|=4$)
or three annuli
(in the case of $|H_1\cap H_2|=6$).
 By Lemma \ref{lem:H},
$H_2 \cap V_2$ is $t_2$-compressible or $t_2$-$\bdd$-compressible
in $(V_2, t_2)$.
 In the former case,
by compressing $H_2 \cap V_2$ we obtain a disk $C$
which is disjoint from $t_2$
and
bounded by a loop of $H_1 \cap H_2$.
 Hence the disks $E$ and $C$ show that
$H_1$ is weakly $K$-reducible again.
This is the conclusion (1).
We consider the latter case.
 Let $A$ be the annulus
to which a $t_2$-$\bdd$-compressing disk of $H_2 \cap V_2$
is incident.
 We can obtain a disk $C'$ from $A$ by $t_2$-$\bdd$-compression.
 If $C'$ is parallel to a subdisk of $\bdd V_2$ in $V_2 - t_2$,
then $A$ is also parallel to a subsurface of $\bdd V_2$
in $V_2 - t_2$,
and the parallelism intersects $H_2$ only in
$A$.
 We can isotope
 $H_2$
along this parallelism
to cancel the two intersection loops $\bdd A$, 
and obtain the conclusion (2).
 Hence we can assume
that $C'$ is a $K$-compressing disk of $H_1$ in $(V_2, t_2)$.
 We can isotope $C'$ slightly
so that $\bdd C'\cap(H_1\cap H_2)=\emptyset$
since $C'$ is obtained by $\bdd$-compression on $H_2 \cap V_2$.
 Hence the disks $E$ and $C'$ show
that
$H_1$ is weakly $K$-reducible.
This is the conclusion (1).
$\qed$

\section{$|H_1\cap H_2|=4$, Case (B)}

 We consider in this section 
Subcase (B) of the case $|H_1\cap H_2|= 4$.
See Figure \ref{fig:l=4}.
In this case, 
$H_{2}-(H_1\cap H_2)$
consists of a disk with three holes $S$
and two annuli $A_1$ and $A_2$.


\begin{proposition}\label{prop:4}
 Under the above condition,
the conclusion $(a), (b)$ or $(c)$ of Theorem \ref{thm:general} holds.
\end{proposition}


To prove this proposition, we need some lemmas. 

We can assume without loss of generality
that the disk with three holes $S$
is 
in 
$V_1$,
and the annuli $A_1$ and $A_2$
are 
in 
$V_2$. 

We have four cases as illustrated in Figure \ref{fig:l=4}.
In either case,
$H_1$ intersects $W_2$
in two annulus components $F_1$ and $F_2$ 
since 
$K$ is contained in 
$W_1$.

\begin{lemma}\label{lem:4-1}
 Suppose
that a component $\delta$ of $\bdd S$ bounds a disk $Q_1$ in $V_1$
such that $Q_1$ intersects the trivial arc $t_1$
in zero or one point.
$($The interior of the disk $Q_1$ may intersect $S$
in a union of finitely many loops.$)$
 Then at least one of the two conditions below holds.
\begin{enumerate}
\renewcommand{\labelenumi}{(\theenumi)}
\item
 The conclusion $(b)$ of Theorem \ref{thm:general} holds.
\item
 The conclusion $(a)$ of Theorem \ref{thm:general} holds.
 Moreover, we can decrease the number of loops of $H_1 \cap H_2$ by two.
\end{enumerate}
\end{lemma}

\begin{proof}
 By Lemma \ref{lem:H},
the union of annuli $A_1\cup A_2$
is $t_2$-compressible or $t_2$-$\bdd$-compressible in $(V_2, t_2)$.
 If $A_1\cup A_2$ is $t_2$-compressible,
then a component of $H_1\cap H_2$ bounds a disk $Q_2$ in $V_2$
such that $Q_2$ does not intersect $t_2$.
 Then the disks $Q_1$ and $Q_2$ together show
that 
$H_1$ is weakly $K$-reducible.

 Suppose that $A_1\cup A_2$ is $t_2$-incompressible
and $t_2$-$\bdd$-compressible in $(V_2,t_2)$.
 If $A_1$ or $A_2$ is parallel to a subsurface of $\bdd V_2$
in $V_2 - t_2$,
then we can cancel two intersection loops of $H_1\cap H_2$.
 This is the conclusion (2). 
 Otherwise,
we obtain a $t_2$-compressing disk $Q'_2$ of $H_1$
in $(V_2, t_2)$
with $\bdd Q'_2 \cap \bdd S = \emptyset$ 
by $t_2$-$\bdd$-compression
on $A_1$ or $A_2$. 
Then 
$Q_1$ and $Q'_2$ show
that $H_1$ is weakly $K$-reducible. 
\end{proof}

\begin{lemma}\label{lem:ABoundaryCompressible}
 Suppose that 
$A_1 \cup A_2$
is $t_2$-$\bdd$-compressible in $(V_2, t_2)$.
 Then at least one of the conditions below holds.
\begin{enumerate}
\renewcommand{\labelenumi}{(\theenumi)}
\item
 The conclusion $(c)$ of Theorem \ref{thm:general} holds.
\item
 The conclusion $(a)$ of Theorem \ref{thm:general} holds.
 Moreover, we can decrease the number of loops of $H_1 \cap H_2$ by two.
\end{enumerate}
\end{lemma}

\begin{proof}
 Let $D$ be a $t_2$-$\bdd$-compressing disk of $A_1 \cup A_2$.
 We can assume without loss of generality
that $D$ is incident to $A_1$.
 Suppose first that $D$ is contained in $W_2$.
 Then $D$ is incident
to one of the annuli $H_1 \cap W_2 = F_1 \cup F_2$, say $F_1$, 
then $F_1$ is parallel to $A_1$ in $W_2$.
 Hence we have the conclusion (2).

 Hence we can assume that $D$ is contained in $W_1$.
 Let $R$ be the component of $H_1 \cap W_1$
such that $R$ contains the arc  $D \cap H_1$.
 Since this arc connects distinct components of $\bdd A_1$,
$R$ has two or more boundary loops.
Therefore, $R$ is either an annulus or disk with two holes
rather than a disk or a torus with one hole.

If $R$ is an annulus disjoint from 
$K$,
then this annulus $R$ is parallel to $A_1$ in $W_1$
and the parallelism is disjoint from 
$K$, 
and we obtain the conclusion (2). 

If $R$ is an annulus which intersects $K$ in a single point,
then the torus $A_1 \cup R$ intersects 
$K$ transversely
in a single point,
and hence it forms a non-separating torus
when pushed slightly into the interior of the handlebody $W_1$.
 This is a contradiction.

 We consider the case
where $R$ is an annulus
which intersects 
$K$ in precisely two points.
We call the annulus $R_1$ (Case (3) in Figure \ref{fig:l=4}). 
 Then $H_1 \cap H_2$ does not contain an inessential loop in $H_1$
(ignoring 
$K \cap H_1$),
and the other component of $H_1 \cap W_1$ is an annulus, say $R_2$,
which is disjoint from 
$K$.
 Since the torus $H_1$ is connected,
each of the annuli $H_1 \cap W_2 = F_1 \cup F_2$ connects
a component of $\bdd R_1$ and a component of $\bdd R_2$,
and hence
a component of $\bdd A_1$ and a component of $\bdd A_2$.
 The union of annuli $F_1 \cup F_2$ is
compressible or $\bdd$-compressible in 
$W_2$
because an incompressible and $\bdd$-incompressible surface
properly embedded in a handlebody is a disk.
We suppose first
that it is compressible.
Compressing $F_1 \cup F_2$,
we obtain two disks in $W_2$,
one of which is bounded by a component of $\bdd A_1$,
and the other 
by a component of $\bdd A_2$.
 The annulus $R_2$ is $K$-compressible or $K$-$\bdd$-compressible
in $(W_1, K)$ by Lemma \ref{lem:SurfaceIn20}.
 In the former case, $K$-compressing $R_2$,
we obtain a disk
which is disjoint from $K$
and is bounded by a component of $\bdd A_2$.
 Then 
$H_2$ is 
weakly $K$-reducible.
This is the conclusion (1).
 Hence we can assume
that $R_2$ is $K$-$\bdd$-compressible in $(W_1, K)$.
 If a $K$-$\bdd$-compressing disk of $R_2$ is incident to $A_2$,
then $R_2$ is parallel to $A_2$ in $W_1$, 
and 
we have the conclusion (2).
 When a $K$-$\bdd$-compressing disk of $R_2$
is incident to $S \cup A_1$,
by performing a $K$-$\bdd$-compressing operation on $R_2$,
we obtain a $K$-compressing disk of $H_2$.
 An adequate small isotopy moves this disk
to be disjoint from $\bdd A_2$.
 Hence 
$H_2$ is weakly $K$-reducible.
 This is the conclusion (1).
 Thus, 
we can assume
that $F_1 \cup F_2$ is $\bdd$-compressible in $W_2$.
 Let $D_2$ be a $\bdd$-compressing disk.
 We can assume without loss of generality
that $D_2$ is incident to $F_1$. 
 Since $F_1$ connects a component of $\bdd A_1$
and a component of $\bdd A_2$,
the arc $\bdd D_2 \cap H_2$ is contained
in the disk with three holes $S$.
The $t_2$-$\bdd$-compressing disk $D$ of $A_1 \cup A_2$ intersects
$A_1$ in a subarc of $\bdd D$,
and this arc is essential in $A_1$
by Definition \ref{def:bdd-comp}.
 We take the disks $D_2$ and $D$
so that they are disjoint from each other.
First, we isotope $H_1$ along 
$D$.
Then 
$A_1$ is deformed into a disk $A'_1$ 
and the disk with three holes $S$ is deformed
into a torus with three holes $S'$.
 The annuli $F_1 \cup F_2$ are connected by a band
and are deformed into a disk with two holes $P$.
 The annulus $R_1$ is deformed into a disk $R'$ 
which intersects 
$K$ in precisely two points.
 The disk $D_2$ now forms a $\bdd$-compressing disk of $P$.
 The arc $\bdd D_2 \cap H_2$ connects
the loop $\bdd A'_1$ and a component of $\bdd A_2$.
Then we isotope $H_1$ along 
$D_2$.
The disk $A'_1$ and the annulus $A_2$ are connected
by a band and deformed into an annulus.
 The torus with three holes $S'$ is deformed
into a torus with two holes.
 The disk with two holes $P$ is deformed
into an annulus.
 The disk $R'$ and the annulus $R_2$ are connected by a band,
and deformed into an annulus
which intersects 
$K$ transversely in two points.
 Thus we have isotoped $H_1$ and $H_2$
so that they intersects in two loops
which are $K$-essential both in $H_1$ and in $H_2$.
 This is the conclusion (2).

 If $R$ is a disk with two holes,
then, by performing a $K$-$\bdd$-compressing operation on $R$
along the disk $D$,
we obtain an annulus $F$.
Note that $F$ is disjoint from 
$K$,
and a component of $\bdd F$ is an inessential loop $\ell$ in $A_1$
and the other component of $\bdd F$ is a component of $\bdd A_2$.
 By gluing $F$ and a disk on $A_1$ along $\ell$, 
we can obtain a disk $F'$
which is bounded by a component of $\bdd A_2$.
This disk $F'$ is disjoint from $K$.
 The union of annuli $H_1 \cap W_2 = F_1 \cup F_2$
is compressible or $\bdd$-compressible in $W_2$.
 When it is compressible,
we obtain a disk bounded by a loop of $\bdd A_2 \subset H_1 \cap H_2$
by compressing $F_1 \cup F_2$.
This disk shows together with $F'$
that 
$H_2$ is (weakly) $K$-reducible.
 This is the conclusion (1) of this lemma.
 Hence we can assume
that $F_1 \cup F_2$ is $\bdd$-compressible in $W_2$.
 If a $\bdd$-compressing disk of $F_1 \cup F_2$ is incident
to one of the annuli $A_1$ and $A_2$, say $A_1$,
then $F_1$ or $F_2$ is parallel to $A_1$ in $W_2$.
 Thus we 
obtain the conclusion (2).
 Hence we can assume
that a $\bdd$-compressing disk of $F_1 \cup F_2$ is incident
to the disk with three holes $S$.
 Then we obtain a compressing disk of $S$ in $W_2$
by performing a $\bdd$-compressing operation on $F_1 \cup F_2$.
 This disk shows together with the disk $F'$
that 
$H_2$ is weakly $K$-reducible.
 This is the conclusion (1) of this lemma.
\end{proof}

\begin{lemma}\label{lem:4-2}
Suppose that $S$ has a $t_1$-compressing disk $D$ in $(V_1,t_1)$.
If $D$ is contained in the handlebody $W_1$,  
then at least one of the two conditions below holds.
\begin{enumerate}
\renewcommand{\labelenumi}{(\theenumi)}
\item
 The conclusion $(c)$ of Theorem \ref{thm:general} holds.
\item
 The conclusion $(a)$ of Theorem \ref{thm:general} holds.
 Moreover, we can decrease the number of loops of $H_1 \cap H_2$ by two
\end{enumerate}
\end{lemma}

\begin{proof}
 The union of annuli $H_1 \cap W_2 = F_1 \cup F_2$
is compressible or $\bdd$-compressible in $W_2$.
 If $F_1 \cup F_2$ is compressible in $W_2$,
by compressing $F_1 \cup F_2$
we obtain an essential disk $E$ in $W_2$
such that $\bdd E \subset H_1\cap H_2$.
 Since $D \subset W_1$, 
the disks $D$ and $E$ show
that $H_2$ is weakly $K$-reducible. 
 This is the conclusion (1).

 Hence we may assume that $F_1 \cup F_2$
is $\bdd$-compressible in $W_2$,
and let $D_2$ be a $\bdd$-compressing disk of $F_1 \cup F_2$.
 If the arc $\bdd D_2 \cap H_2$ is contained
in one of the annuli $A_1$ and $A_2$, say $A_1$,
then $F_1$ or $F_2$ is parallel to $A_1$ in $W_2$,
we have the conclusion (2).

 Suppose that the arc $\bdd D_2 \cap H_2$ is contained in $S$.
 By $\bdd$-compression along $D_2$, 
we obtain from $F_1$ or $F_2$ an essential disk $D'_2$ in $W_2$.
 Note that $\bdd D'_2 \subset S$.
 (The boundary loop $\bdd D$ may intersect $\bdd D'_2$.)
 The union of annuli $A_1\cup A_2$ is
$t_2$-compressible or $t_2$-$\bdd$-compressible in $(V_2,t_2)$
by Lemma \ref{lem:H}.
 When it is $t_2$-$\bdd$-compressible,
we obtain the desired conclusion
by Lemma \ref{lem:ABoundaryCompressible}.
 Hence we may assume
that $A_1 \cup A_2$ has a $t_2$-compressing disk $G$
in $(V_2, t_2)$.
 Since the interior of $G$ is disjoint from $H_2$,
it is entirely contained in $W_1$ or $W_2$.
 If $G$ is in $W_1$,
then $G$ and $D'_2$ assure that $H_2$ is weakly $K$-reducible.
 If $G$ is in $W_2$,
then $G$ and $D$ assure that $H_2$ is weakly $K$-reducible.
 In both cases, we obtain the conclusion (1).
\end{proof}

\begin{lemma}\label{lem:4-3}
 Suppose that $S$
has a $t_1$-compressing disk $D$ in $(V_1, t_1)$.
 If $D$ is contained in the handlebody $W_2$,
then at least one of the two conditions below holds.
\begin{enumerate}
\renewcommand{\labelenumi}{(\theenumi)}
\item
 The conclusion $(c)$ of Theorem \ref{thm:general} holds.
\item
 The conclusion $(a)$ of Theorem \ref{thm:general} holds.
\end{enumerate}
\end{lemma}

\begin{proof}
 By Lemma \ref{lem:H},
$A_1 \cup A_2$
is $t_2$-compressible or $t_2$-$\bdd$-compressible
in $(V_2, t_2)$.
 In the latter case,
we are done in Lemma \ref{lem:ABoundaryCompressible}.
 Hence we may assume
that $A_1 \cup A_2$ is $t_2$-compressible.
 Let $D_1$ be a $t_2$-compressing disk of $A_1 \cup A_2$.
 If $D_1$ is in $W_1$,
then this disk $D_1$ shows together with $D$ that 
$H_2$ is weakly $K$-reducible.
 Thus we obtain the conclusion (1).

 Hence we may assume that $D_1$ is in $W_2$.
 Here we have four cases (see Figure \ref{fig:l=4}) on $H_1\cap W_1$:
(1) $H_1\cap W_1$ consists of
a disk $Q$ which intersects $K$ in two points,
an annulus $R$ which is disjoint from $K$
and a torus with one hole which is disjoint from $K$;
(2) $H_1\cap W_1$ consists of
a disk $Q$ which intersects $K$ in two points
and a disk with two holes $R$ which is disjoint from $K$; 
(3) $H_1\cap W_1$ consists of
an annulus $R_1$ which intersects $K$ in two points
and an annulus $R_2$ which is disjoint from $K$; and
(4) $H_1\cap W_1$ consists of two annuli $R_1, R_2$
each of which intersects $K$ in one point.

{\bf Case (1) or (4).}
Set $J=Q \cup R$ (Case (1)) or $R_1\cup R_2$ (Case (4)).
 By Lemma \ref{lem:SurfaceIn20}, 
$J$ is $K$-compressible or $K$-$\bdd$-compressible
in $(W_1, K)$.
 If $J$ is $K$-compressible,
then, by compressing $J$,
we obtain a disk $D_2$ 
such that $D_2 \cap K=\emptyset$
and that $\bdd D_2$ is a component of $H_1\cap H_2$.
 The disks $D_2$ and $D$ assure
that $H_2$ is weakly $K$-reducible. 
 This is the conclusion (1) of this lemma.

 Suppose that $J$ is $K$-$\bdd$-compressible in $(W_1, K)$,
and let $D_3$ be a $K$-$\bdd$-compressing disk of $J$.
 Let $J_0$ be the component of $J$
to which the $K$-$\bdd$-compressing disk $D_3$ is incident.
 If the arc $\bdd D_3 \cap H_2$ is contained in $A_1 \cup A_2$,
then the disk $D_3$ is also a $t_2$-$\bdd$-compressing disk
of $A_1 \cup A_2$ in $(V_2, t_2)$,
and we are done in Lemma \ref{lem:ABoundaryCompressible}.
 Hence we may assume
that the arc $\bdd D_3 \cap H_2$ is contained
in 
$S$.
 (The arc $\bdd D_3 \cap H_2$ may intersect a loop of $H_1 \cap H_2$ 
in its interior in case (1).
 However, we can isotope it so that $\bdd D_3 \cap H_2$ is contained in $S$,
since $\bdd D_3 \cap H_2 $ connects the same side of $Q$
if $D_3$ intersects $Q$.)
 By $K$-$\bdd$-compression on $J_0$ along $D_3$
and an adequate small isotopy,
we obtain an essential disk $D'_3$ in $W_1$
such that $\bdd D'_3 \cap D_1 = \emptyset$
and that $D'_3$ intersects 
$K$ transversely
in at most one point.
The disks $D_1$ and $D'_3$ show
that 
$H_2$ is weakly $K$-reducible.
Thus we obtain the conclusion (1).

{\bf Case (3).}
We may assume without loss of generality 
that the $t_2$-compressing disk $D_1$ is incident to $A_1$
rather than to $A_2$.
 Suppose first
that a component of $\bdd R_2$
is a component of $\bdd A_1$.
 By Lemma \ref{lem:SurfaceIn20}, 
$R_2$ is $K$-compressible or $K$-$\bdd$-compressible in $(W_1,K)$.
 If $R_2$ is $K$-compressible,
then $K$-compressing on $R_2$ yields a disk
which is disjoint from 
$K$
and is bounded by a loop of $H_1 \cap H_2$.
Hence this disk and 
$D_1$ together
show that 
$H_2$ is weakly $K$-reducible.
This is the conclusion (1). 
Suppose that $R_2$ is $K$-$\bdd$-compressible in $(W_1,K)$.
Let $D_4$ be a $K$-$\bdd$-compressing disk of $R_2$.
If the arc $\bdd D_4 \cap H_2$ is contained in $A_1$,
then $R_2$ is parallel to $A_1$ in $W_1$,
and the parallelism is disjoint from 
$K$.
We obtain the conclusion (2). 
 If the arc $D_4$ is not entirely contained in $A_1$,
then, by $K$-$\bdd$-compression on $R_2$ along $D_4$,
we obtain a disk
which is essential in $W_1$ and is disjoint from 
$K$.
An adequate isotopy moves this disk
to be disjoint from 
$\bdd D_1$
which is parallel to $\bdd A_1$ in $H_2$.
This disk and the disk $D_1$ show
that 
$H_2$ is weakly $K$-reducible.
We obtain the conclusion (1) again.

 Hence we may assume
that $\bdd R_2 = \bdd A_2$ and $\bdd R_1 = \bdd A_1$.
Then each of the annuli $H_1 \cap W_2 = F_1 \cup F_2$
is bounded by a component of $\bdd A_1$ and that of $\bdd A_2$.
 Recall that the annulus $A_1$ has a $t_2$-compressing disk $D_1$
in $W_2 \cap V_2$.
Performing the $t_2$-compression on a copy of $A_1$,
we obtain a disk
which is contained in $W_2 \cap V_2$
and shows that $F_1$ is compressible.
 By compressing $F_1$,
we obtain a compressing disk of 
$A_2$ in $W_2 \cap V_2$.
Since $\bdd R_2 = \bdd A_2$ and $A_2$ has a compressing disk in $W_2$,
we can apply similar argument
as in the previous paragraph to this case,
to obtain the desired conclusion.

{\bf Case (2).}
 By Lemma \ref{lem:SurfaceIn20}, 
$Q \cup R$ is $K$-compressible or $K$-$\bdd$-compressible in $(W_1, K)$.
 If $Q\cup R$ is $K$-compressible,
then 
$K$-compression on $Q \cup R$ yields a disk
which is disjoint from $K$. 
 This disk is bounded by a loop of $H_1 \cap H_2$,
and assures that $H_2$ is weakly $K$-reducible
together with $D$.
 This is the conclusion (1).

So, we may assume that $Q \cup R$ is $K$-$\bdd$-compressible.
Let $D_5$ be a $K$-$\bdd$-compressing disk of $Q \cup R$.
 If $D_5$ is incident to 
$Q$, 
then we obtain a disk $Q'$
from $Q$ by $K$-$\bdd$-compression along $D_5$.
 Note that $Q'$ intersects $K$ in a single point.
 The disks $D_1$ and $Q'$ show
that 
$H_2$ is weakly $K$-reducible.
This is the conclusion (1).

Thus we may assume
that $D_5$ is incident to 
$R$.
If the arc $\bdd D_5 \cap H_2$ is contained in $A_1$ or $A_2$,
then the disk $D_5$ is also a $t_2$-$\bdd$-compressing disk
of $A_1 \cup A_2$,
and we obtain the desired conclusion
by Lemma \ref{lem:ABoundaryCompressible}.
 Hence we can assume
that the arc $\bdd D_5 \cap H_2$ is contained
in 
$S$.
If the arc $\bdd D_5 \cap H_2$
connects two distinct components of $\bdd R$,
then we can reduce the number of intersection loops $H_1 \cap H_2$
by isotoping $H_1$ along the disk $D_5$,
and obtain the conclusion (2).  

Therefore we may assume
that the arc $\bdd D_5 \cap H_2$ has its two endpoints
in a single component, say $\ell$, of $\bdd S$.
 The loop $\ell$ is a component of $\bdd A_j$ for $j=1$ or $2$.
 Set $k$ so that $\{ j, k \} = \{ 1, 2 \}$.
We isotope 
$H_1$ along 
$D_5$.
Then the disk with two holes $R$ is deformed
into two annuli $R_1$ and $R_2$.
 For $i=1$ and $2$,
let $\ell_i$ be a component of $\bdd R_{i}$
such that $\ell_{i}\not\subset\bdd R$.
See Figure \ref{fig:4b-1}.

\begin{figure}[htbp]
\centering
\includegraphics[width=.7\textwidth]{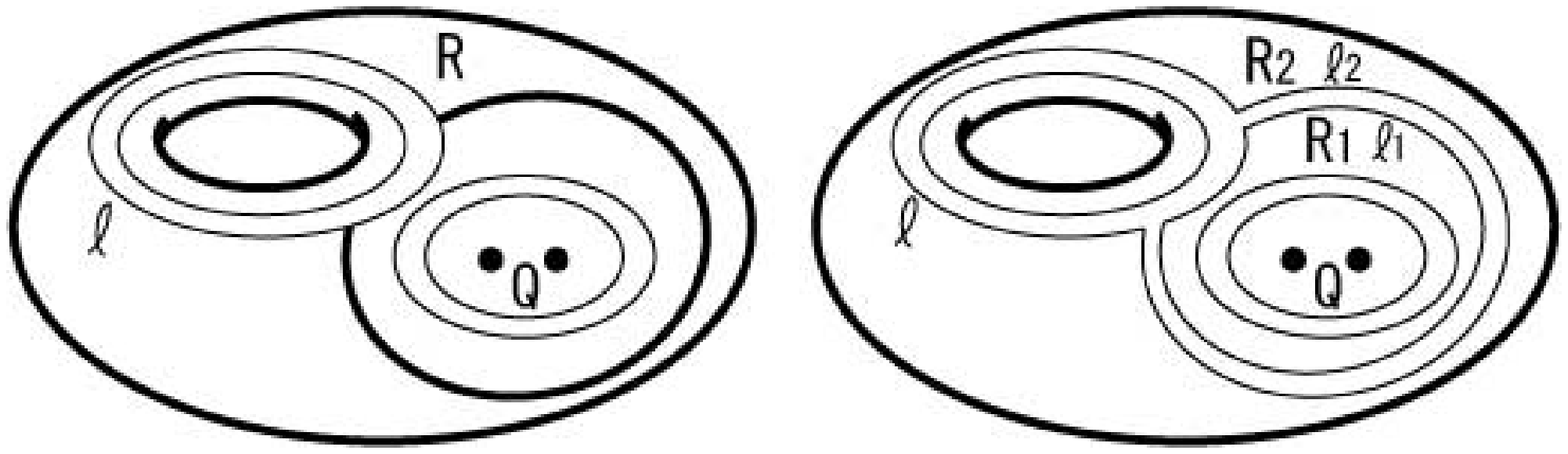}
\caption{}
\label{fig:4b-1}
\end{figure}

 Then the annulus $A_j$ is deformed
into a disk with three holes $A'_j$
and the disk with three holes $S$ is deformed
into two planar surfaces $S_1$ and $S_2$,
where $S_i$ contains the loop $\ell_i$
for $i=1$ and $2$.
 One of $S_1$ and $S_2$ is an annulus,
and the other is a disk with two holes.
 We can assume that $S_1$ is a disk with two holes,
by changing the suffix numbers of $R_1$ and $R_2$
if necessary.
 Now, $H_1 \cap H_2$ consists of five loops
which are $K$-essential both in $H_1$ and 
$H_2$.

The union of the disk $Q$ and the annuli $R_1$ and $R_2$
is $K$-compressible or $K$-$\bdd$-compressible in $(W_1, K)$
by Lemma \ref{lem:SurfaceIn20}.
In the former case,
$K$-compressing yields a disk
which is disjoint from 
$K$.
This disk is bounded by one of the five loops $H_1 \cap H_2$,
and 
then its boundary loop is contained
in 
$S$.
 Thus this disk and 
$D_1$ together show
that 
$H_2$ is weakly $K$-reducible. 

We consider the latter case.
Let $D_6$ be a $K$-$\bdd$-compressing disk of $Q \cup R_1 \cup R_2$.
If the arc $\bdd D_6 \cap H_2$ is contained in the annulus $A_k$,
then $D_6$ forms a $K$-$\bdd$-compressing disk of $Q \cup R$
before the boundary compression on $R$.
 Hence we obtain the desired conclusion
by Lemma \ref{lem:ABoundaryCompressible}.
If the arc $\bdd D_6 \cap H_2$ is contained in 
the annulus $S_2$,
then the disk $D_6$ is incident to the annulus $R_2$. 
Hence 
$R_2$ is parallel to 
$S_2$,
and 
we can isotope $H_1$ so that 
$H_1$ and $H_2$ intersect each other in three loops
which are $K$-essential both in $H_1$ and in $H_2$.
 We obtain the conclusion (2).

Suppose that the arc $\bdd D_6 \cap H_2$ is contained
in 
$S_1$.
Then 
$D_6$ is incident to 
$Q$ or 
$R_1$.
(If $D_6$ were incident to 
 $R_2$,
then it would be incident to 
$\ell_2$
which is a component of $\bdd S_2$.)
 We perform $K$-$\bdd$-compressing on $Q \cup R_1$ along $D_6$,
to obtain a disk
which intersects 
$K$ in at most one point.
 Since $S_1$ is a disk with two holes,
this disk can be isotoped
to be bounded by a loop of $\bdd S_1$.
 Hence this disk together with $D_1$ shows
that 
$H_2$ is weakly $K$-reducible.
This is the conclusion (1).

 Suppose that the arc $\bdd D_6 \cap H_2$ is contained
in the disk with two holes $A'_j$.
 We perform $K$-$\bdd$-compression
on $Q \cup R_1 \cup R_2$ along $D_6$,
to obtain a disk
which intersects 
$K$ in at most one point.
 Since $A'_j$ is a disk with two holes,
this disk can be isotoped
to be bounded by a loop of $\bdd A'_j$.
Hence this disk together with $D_1$ shows
that 
$H_2$ is weakly $K$-reducible.
 This is the conclusion (1).
\end{proof}

\medskip

\noindent{\bf Proof of Proposition \ref{prop:4}.}
 By Lemma \ref{lem:H},
the disk with three holes $S$
is $t_1$-compressible or $t_1$-$\bdd$-compressible in $(V_1, t_1)$.
 If $S$ is $t_1$-compressible in $(V_1, t_1)$,
then we are done by Lemmas \ref{lem:4-2} and \ref{lem:4-3}.
 Hence we may assume
that $S$ is $t_1$-incompressible and $t_1$-$\bdd$-compressible,
and we have three cases (1)--(3) in Lemma \ref{lem:4-holes}.

In Case (1), we are done by Lemma \ref{lem:4-1}.
In Case (2) we obtain the conclusion 
 of Theorem \ref{thm:general}.

 We consider Case (3).
 $S$ is $t_1$-incompressible and meridionally compressible
in $(V_1, t_1)$.
 The union of annuli $H_1 \cap W_2 = F_1 \cup F_2$
is compressible or $\bdd$-compressible in the handlebody $W_2$.
 If it is compressible in $W_2$,
then the compression of $F_1 \cup F_2$ yields a disk $E_2$ 
which is bounded by a loop of $\bdd S = H_1 \cap H_2$.
 Since we are considering Case (3),
there is a meridionally compressing disk $E_1$ of $S$, 
and $E_1$ is contained in $W_1$
because it intersects 
$K$.
 Then the disks $E_1$ and $E_2$ show that 
$H_2$ is weakly $K$-reducible.
 This is the conclusion 
(c) of Theorem \ref{thm:general}.

 Suppose that $F_1 \cup F_2$ is $\bdd$-compressible in $W_2$.
 Let $D'$ be a $\bdd$-compressing disk.
 If the arc $\bdd D' \cap \bdd W_2$ is in either $A_1$ or $A_2$,
then either $F_1$ or $F_2$ is parallel to one of $A_1$ and $A_2$.
 The conclusion (a) of Theorem \ref{thm:general} holds.
 Suppose that 
$\bdd D' \cap \bdd W_2$ is in $S$.
 We obtain a $t_1$-compressing disk of $S$ from $F_1 \cup F_2$
by 
$\bdd$-compression along $D'$.
 This contradicts the assumption.

This completes the proof of Proposition \ref{prop:4}.
$\qed$

\medskip

\noindent{\bf Proof of Theorem \ref{thm:general} (1).}
Propositions \ref{prop:general}, 
\ref{prop:456}
and \ref{prop:4} show Theorem \ref{thm:general} (1).
$\qed$



\section{When $|H_1\cap H_2|=1$}

 We consider in this section the case where $H_1\cap H_2$
consists of a single $K$-essential loop, say $l$.
 We will show Theorem 
\ref{thm:general} (3-2).
 We will use the condition
that $M$ has a $2$-fold branched cover with branch set $K$
in the proofs of Lemmas \ref{lem:Qcompressible}
and \ref{lem:bothincompressible}
only when applying Proposition \ref{thm:koba}.

 Since $H_2$ separates $M$,
the loop $l$ is inessential in $H_1$,
and separates the $(1,1)$-splitting torus $H_1$
into a disk, say $Q$, and a torus with one hole, say $H'_1$.
 $Q$ intersects $K$ in two points, 
and is contained in
$W_1$. 
 The torus with one hole $H'_1$ is contained in
$W_2$.
Since $H_1$ separates $M$, 
the loop $l$ separates $H_2$ into two tori with one hole.
See Figure \ref{fig:l=1}.

\begin{figure}[htbp]
\centering
\includegraphics[width=.7\textwidth]{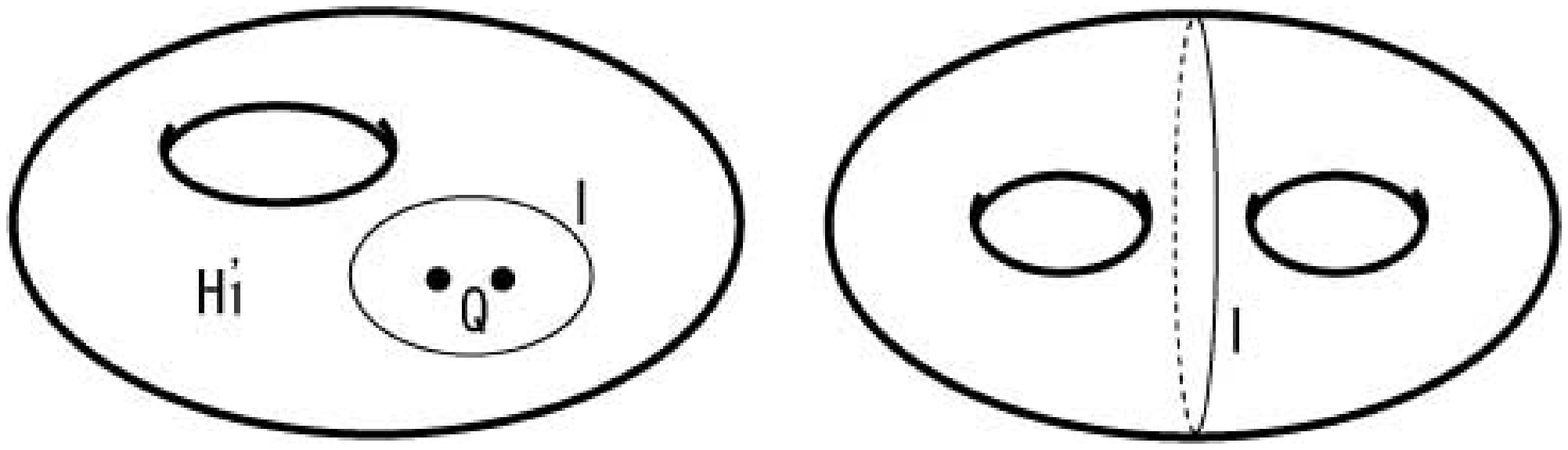}
\caption{}
\label{fig:l=1}
\end{figure}

 The torus with one hole $H'_1$ is compressible or $\bdd$-compressible
in the handlebody $W_2$.
By Lemma \ref{lem:SurfaceIn20},
the disk $Q$ is $K$-compressible or $K$-$\bdd$-compressible in $(W_1,K)$.

\begin{lemma}\label{lem:H'_1compressible}
 If $H'_1$ is compressible
in
$W_2$,
then 
the conclusion $(c)$ of Theorem \ref{thm:general} holds.
\end{lemma}

\begin{proof}
 Compressing $H'_1$, 
we obtain a disk, say $D_2$ in $W_2$ with $\bdd D_2 = l$.
 When $Q$ is $K$-compressible,
compression on $Q$ yields a disk $D_1$ 
disjoint from $K$ with $\bdd D_1 = l$.
 Then $D_1$ and $D_2$ show that $H_2$ is weakly $K$-reducible.
 When $Q$ is $K$-$\bdd$-compressible,
$\bdd$-compression on $Q$ yields
a disk $D'_1$ intersecting $K$ at a single point.
 We can isotope $D'_1$ so that $\bdd D'_1 \cap l = \emptyset$.
 Then $D'_1$ and $D_2$ show that $H_2$ is weakly $K$-reducible.
%
%
\end{proof}

\begin{lemma}\label{lem:Qcompressible}
Suppose that
$Q$ is $K$-compressible in $(W_1, K)$,
and that $H'_1$ is incompressible in $W_2$.
 Then the conclusion $(c)$, $(d)$ or $(g)$ of Theorem \ref{thm:general} holds.
\end{lemma}

\begin{proof}
Let $D_0$ be a $K$-compressing disk of $Q$ in $(W_1, K)$.
Performing a $K$-compressing operation on
$Q$ along $D_0$,
we obtain a disk $Q'$ 
with 
$Q'\cap K=\emptyset$
and 
$\bdd Q'=l$.
Then $Q'$ is a $K$-compressing disk of $H_2$.
Note that $D_0$ and a subdisk of $Q$
bounds a $3$-ball $B$,
and $Q$ and $Q'$ are isotopic in $W_1$ ignoring
$K$.
Hence the torus $H'_1 \cup Q'$ is a Heegaard splitting surface of $M$.
The disk $Q'$ divides the handlebody $W_1$
into two solid tori, say $U_1$ and $U_2$,
where $U_1$ contains
$K$ as a core
(Lemma 3.3 in \cite{GHY}).
Let $H_{2i}$ be the torus with one hole $H_2 \cap U_i$
for $i=1$ and $2$.

Since
$H'_1$ is incompressible in
$W_2$,
it 
has a $\bdd$-compressing disk $D$.

 First we consider the case
where the arc $\bdd D \cap H_2$ is contained in $H_{21}$.
By
Definition \ref{def:bdd-comp},
the arc $\bdd D \cap H_2$ is an essential arc
in
$H_{21}$.
We isotope the torus $H'_1 \cup Q'$ along the disk $D$.
By this isotopy
$(H'_1 \cup Q') \cap W_1$ and $(H'_1 \cup Q')\cap W_2$ are 
deformed into annuli,
say $A_1$ and $A_2$ respectively.
 The annulus $A_1$ cuts $W_1$
into a solid torus $U'_1$ and a handlebody $U'_2$ of genus $2$.
 Note that $U'_1$ contains
 $K$ as a core.
 The annulus $A_2$ is incompressible in $W_2$
since $H'_1$ is incompressible in $W_2$.
 Hence $A_2$ cuts off a solid torus $U_3$ from $W_2$
since $l$ is separating on $H_2$ (see, for example, Section 3 in \cite{K0}).
 Note that $U'_1 \cap U_3 = \bdd U'_1 \cap H_2$.
 This is an annulus, which we will call $A$.
If the loops of 
$\bdd A_2$ are meridians of $U_3$,
then $A_2$ is compressible in $W_2$,
which is a contradiction.
If the boundary loops $\bdd A_2$ are longitudes of $U_3$,
then $A_2$ is parallel to the annulus $A$ in $W_2$,
and hence
$H'_1$ is parallel to $H_{21}$ in $W_2$
before the isotopy of $H'_1 \cup Q'$ along
$D$.
 We can isotope $H'_1$ onto $H_{21}$ 
and then slightly into $\text{int}\,W_1$.
 Then the torus $H'_1 \cup Q'$
bounds a solid torus which contains
$K$ as a core.
Remember that it is a Heegaard splitting torus of $M$.
 Since the other side of this torus is also a solid torus,
$K$ is a core knot in a lens space
or the trivial knot in the $3$-sphere.
 Hence $H_2$ is $K$-reducible 
(Proposition 4.6).
 This is the conclusion 
(c) of Theorem \ref{thm:general}.
 Thus we can assume that the boundary loops $\bdd A_2$
are not meridians or longitudes of the solid torus $U_3$.
 If the boundary loops $\bdd A_1$
are not longitudes of the solid torus $U'_1$,
then the $3$-manifold $U'_1 \cup U_3$ is not a solid torus,
which contradicts
that it is bounded by the Heegaard splitting torus $H'_1 \cup Q'$ of $M$.
 Hence the loops of $\bdd A_1$ are longitudes of $U'_1$,
and the annulus $A_1$ is parallel to the annulus $A$ in $W_1$
ignoring
$K$.
Since
$K$ is a core of the solid torus $U'_1$,
it is parallel into the annulus $A_1$ in $W_1$,
and hence into the Heegaard splitting torus $H'_1 \cup Q'$.
 Hence $K$ is the trivial knot, a core knot or a torus knot.
 This implies the conclusion 
(c) or (d) of Theorem \ref{thm:general}.

 We consider the case
where the arc $\bdd D \cap H_2$ is contained
in the torus with one hole $H_{22}$.
 We may
assume without loss of generality
that the torus with one hole $H_{21}$ is contained
in the solid torus $V_1$
 By Lemma \ref{lem:H},
$H_{21}$ is $t_1$-compressible or $t_1$-$\bdd$-compressible
in
$(V_1, t_1)$.
 In the former case,
by compressing $H_{21}$, 
we obtain a $K$-compressing disk of
$Q$ in $V_1$.
Since $Q$ has another $K$-compressing disk $D_0$ in $W_1 \cap V_2$,
$H_1$ is $K$-reducible.
Therefore
$K$ is trivial
(Theorem B in \cite{Hy1}),
and
$H_2$ is $K$-reducible
(Proposition 4.6).
This is the conclusion 
(c) of Theorem \ref{thm:general}.
 In the latter case,
let $D'$ be a $\bdd$-compressing disk of $H_{21}$ in $(V_1, t_1)$.
 If $D'$ is contained in $V_1 \cap W_2$,
then it is also a $\bdd$-compressing disk of $H'_1$
by Definition \ref{def:bdd-comp}.
We have considered this case in the previous paragraph.
Hence we may assume that the disk $D'$ is contained in $V_1 \cap W_1$.
Then $D'$ is a $K$-$\bdd$-compressing disk of the disk $Q$ in $(W_1, K)$.
We isotope the torus $H'_1 \cup Q'$
along the $\bdd$-compressing disk $D$ of $H'_1$
(rather than along 
$D'$ 
).
 By this isotopy
$(H'_1 \cup Q') \cap W_1$ and $(H'_1 \cup Q')\cap W_2$ are 
deformed into annuli,
say $R_1$ and $R_2$ respectively.
 Note that
 $R_2$ is incompressible in $W_2$
since
$H'_1$ is incompressible in $W_2$.
The annulus $R_1$ cuts $W_1$
into a solid torus $U^*_2$ and
genus 2 handlebody $U^*_1$.
Note that $U^*_1$ contains
$K$ as a core.
The annulus $R_2$ also cuts off a solid torus $U_4$ from $W_2$.
Note that $U^*_2 \cap U_4 = \bdd U^*_2 \cap H_2$.
 This is an annulus, which we will call $R$.
If
the loops of $\bdd R_2$ are meridians of $U_4$,
then $R_2$ is compressible in $W_2$,
which is a contradiction.
If
the loops of $\bdd R_2$ are longitudes of $U_4$,
then $R_2$ is parallel to
$R$ in $W_2$,
and hence
$H'_1$ is parallel to $H_{22}$ in $W_2$
before the isotopy of $H'_1 \cup Q'$ along
$D$.
We isotope
$H_1$ near $H'_1$ 
so that $H'_1$ is moved onto $H_{22}$
and 
then 
slightly into $\text{int}\,W_1$.
 Moreover,
performing a $K$-$\bdd$-compressing operation
on a copy of the disk $Q$ along the $K$-$\bdd$-compressing disk $D'$,
we obtain a meridionally compressing disk of $H_2$ in $(W_1, K)$.
 We can isotope this disk slightly off of $Q$ in $(W_1, K)$
and hence off of $H_1$.
 Then Proposition \ref{thm:koba} shows
that
$H_2$ is weakly $K$-reducible.
 This is the conclusion 
(c) of Theorem \ref{thm:general}.

Hence we may assume that
the loops of 
$\bdd R_2$ are
neither meridians nor longitudes of
$U_4$.
Thus the loops of $\bdd R_1$ are longitudes of $U^*_2$.
We take a core loop $c_0$ of
$R_1$
so that it intersects the $K$-compressing disk $D_0$ of $Q$
in a single arc.
 Let $c$ be the arc $\text{cl}\,(c_0 - D_0)$.
 We take a canceling disk $C_2$ of the arc $t_2 = K \cap B$
in $(V_2, t_2)$
so that $C_2$ is entirely contained in $B$.
 We extend the arc $c$ adding two arcs on the disk $Q$
such that they connects endpoints of $c$ and $\bdd t_2$
and that their interiors are disjoint from the arc $\bdd C_2 \cap Q$.
 Then the extended arc $\gamma$
forms a spine of $(W_1, K)$
and entirely contained in
$H_1$.
 We can isotope the arc $t_2$ along $C_2$
onto the arc $\bdd C_2 \cap Q$.
 Thus we obtain the conclusion 
(g) of Theorem \ref{thm:general}.
\end{proof}

\begin{lemma}\label{lem:bothincompressible}
Suppose that
$Q$ is $K$-incompressible in $(W_1, K)$,
and that
$H'_1$ is incompressible in $W_2$.
Then either one of the two conditions below holds.
\begin{enumerate}
\renewcommand{\labelenumi}{(\theenumi)}
\item
 The conclusion $(c)$ of Theorem \ref{thm:general} holds.
\item
We can isotope the $(1,1)$-splitting torus $H_1$ in $(M,K)$
so that $H_1$ intersects $W_1$ in a separating essential disk
which intersects
$K$ in two points
and is $K$-compressible.
\end{enumerate}
\end{lemma}

We have already studied the situation of the conclusion (2)
in Lemmas \ref{lem:H'_1compressible} and \ref{lem:Qcompressible}.

\begin{proof}
Since $H'_1$ is incompressible in the handlebody $W_2$,
it has a $\bdd$-compressing disk $D_2$ in $W_2$.
 Since $Q$ is $K$-incompressible in $(W_1, K)$,
it has a $K$-$\bdd$-compressing disk $D_1$ in $(W_1, K)$
by Lemma \ref{lem:SurfaceIn20}.
 For $i=1$ and $2$,
let $U_i$ be the solid torus $V_i \cap W_1$,
and $H_{2i}$ the torus with one hole $H_2 \cap V_i$.

 For $i=1$ and $2$,
$H_{2i}$
is $t_i$-compressible or $t_i$-$\bdd$-compressible in $(V_i, t_i)$
by Lemma \ref{lem:H}.
 In the former case,
compression on $H_{2i}$ yields a $K$-compressing disk of $Q$ or $H'_1$.
This contradicts the assumption of this lemma.
Thus $H_{2i}$ is $t_i$-incompressible
and $t_i$-$\bdd$-compressible in $(V_i, t_i)$
for $i=1$ and $2$.

We show that we can take $D_1$ and $D_2$
so that they are separated by
$H_1$.
Suppose that both $D_1$ and $D_2$ are contained in $V_1$, say.
Let $D'$ be a $\bdd$-compressing disk of $H_{22}$ in $(V_2, t_2)$.
 If $D'$ is contained in $W_1$ (resp. in $W_2$),
then
it is also a $K$-$\bdd$-compressing disk of $Q$ (resp. $H'_1$)
by Definition \ref{def:bdd-comp}.
 We can substitute $D'$ for $D_1$ (resp. $D_2$)
so that $H_1$ separates $\bdd$-compressing disks $D_1$ and $D_2$.

 Thus we may assume, without loss of generality,
that $D_1 \subset V_1$ and $D_2 \subset V_2$.
 Let $P=N(\partial Q \cup (Q\cap \partial D_1))$ be a neighborhood
of the union of the boundary loop $\partial Q$
and the arc $Q \cap \partial D_1$ in $Q$.
 Then $P$ is a disk with two holes.
 We can isotope $Q$ along $D_1$
so that $P$ is isotoped into $H_2$.
 The intersection $H_1 \cap \text{int}\,W_1$ is deformed
into two disks each of which intersects
$K$
transversely in a single point.
Let $Q_1$ and $Q_2$ denote the closures of these disks.
 The boundary loops $\partial Q_1$ and $\partial Q_2$ are parallel
on
$H_2$, and
bound an annulus $R_1$ on $H_2$.
Let $B=N(H'_1 \cap \partial D_2)$ be a neighborhood
of the arc $H'_1 \cap \partial D_2$ in $H'_1$.
 We can isotope $H'_1$ along $D_2$
so that the disk $B$ is isotoped into $H_2$
and that $B \cap P = B \cap \partial H'_1$.
 Then after this isotopy
$H_1$ intersects $H_2$ in a $2$-sphere with four holes $S$.
 The intersection $H_1 \cap \text{int}\,W_2$ is deformed
into an annulus,
the closure of which we will call $A_{12}$.
 The boundary loops $\partial A_{12}$ are parallel
on
$H_2$,
and
bound an annulus $R_2$ in $H_2$.

The annulus $A_{12}$ is incompressible in $W_2$
since
$H'_1$ is incompressible. 
Hence $A_{12}$ has a $\bdd$-compressing disk $Z$
in 
$W_2$.
 If the arc $H_2 \cap \partial Z$ is contained in the annulus $R_2$,
then $A_{12}$ is parallel to $R_2$ in $W_2$.
 We isotope
 $H_1$
along the parallelism between $A_{12}$ and $R_2$,
so that $H_1$ is entirely contained in $W_1$.
 We can take a parallel copy of the meridionally compressing disk $Q_1$
of $H_2$ in $(W_1, K)$
so that it is disjoint from $H_1$.
 Then Proposition \ref{thm:koba} shows
that
$H_2$ is weakly $K$-reducible.
 Hence we can assume
that the $\bdd$-compressing disk $Z$ is contained in $W_2 \cap V_1$.
 Performing a $\bdd$-compressing operation on a copy of $A_{12}$
along $Z$,
we obtain a separating essential disk $E$ in $W_2$.
We isotope this disk $E$ slightly off of
$R_2$.

 The annulus $R_2$ is
$t_2$-compressible or $t_2$-$\bdd$-compressible in $(V_2, t_2)$
by Lemma \ref{lem:H}.
 In the former case,
let $X$ be a $t_2$-compressing disk of $R_2$.
 If $X$ is contained in $W_1$,
then the disks $E$ and $X$ together show
that
$H_2$ is weakly $K$-reducible.
 This is the conclusion (1).
 If $X$ is contained in $W_2$,
then the disks $Q_1$ and $X$ together show
that
$H_2$ is weakly $K$-reducible.
This is
the conclusion (1) again.
In the latter case,
let $Y$ be a $t_2$-$\bdd$-compressing disk of $R_2$.
First, suppose that $Y$ is contained in $W_2$.
Then
$A_{12}$ is parallel to
$R_2$ in $W_2$.
 We have considered this situation in the previous paragraph.
Therefore we may assume that $Y$ is contained in $W_1 \cap V_2$.
We can isotope
$Y$ near its boundary loop
so that the arc $\bdd Y \cap H_1$
is entirely contained in the $2$-sphere with four holes $S$,
because $(H_1 \cap W_1) - S$ is a union of the two disks $Q_1$ and $Q_2$
each of which intersects $K$ transversely in a single point.
 There is an arc $\alpha$ on $S$
such that $\alpha$ connects
the two boundary loops $\bdd Q_1$ and $\bdd Q_2$
and that $\alpha$ is disjoint from $\bdd Y$.
 We take a regular neighborhood $N(\alpha)$ of $\alpha$ on $S$,
and isotope the interior of the disk $Q'=Q_1 \cup N(\alpha) \cup Q_2$
slightly into $\text{int}\,W_1$.
 We isotope the remainder part $H_1 - Q'$ slightly
into $\text{int}\,W_2$,
fixing $Y$ with $\bdd Y \subset H_2$. 
 Then, after this isotopy,
$H_1 \cap W_1 = Q'$ the essential separating disk
intersecting
$K$ in two points,
and $Q'$ is disjoint from
$Y$.
 The loop $\bdd Q'$ separates
 $H_2$
into two tori with one hole.
 One of them contains the boundary loop $\bdd Y$,
and gives a $K$-compressing disk of $Q'$ in $(W_1, K)$
when compressed along $Y$.
 Thus we obtain the conclusion (2) of this lemma.
\end{proof}

\medskip

\noindent{\bf Proof of Theorem 
\ref{thm:general} (3-2).}\\
Lemmas \ref{lem:H'_1compressible}, \ref{lem:Qcompressible}
and \ref{lem:bothincompressible}
together show Theorem \ref{thm:general} (3-2).



\section{Examples}\label{sec:example}

We observe an example which realizes the conclusion (2) 
in Theorem \ref{thm:main}. 
This example was informed us by H.J. Song.
 The knot is the Morimoto-Sakuma-Yokota knot of type $(5, 7, 2)$
\cite{MSY}.
Let $K$ be a knot in Figure \ref{fig:572}. 
According to SnapPea, it is hyperbolic.
It is easy to check that $\gamma_1$ and $\gamma_2$ are unknotting 
tunnels of $K$. 
We can see that $\gamma_2$ is isotopic into the $(1, 1)$-splitting torus
as in Figure \ref{fig:572torus}.

Ishihara informed us that the depth of $\gamma_2$ is equal to 2, 
and this implies that $\gamma_2$ is not (1,1)-tunnel  
according to the results by Cho and McCullough \cite{CM1,CM2}.
Ishihara used his algorithm to compute parameters of tunnels \cite{I}.

The two endpoints $\partial \gamma_2$
divide the knot $K$ into two subarcs $K'$ and $K''$
as shown in Figure \ref{fig:572torus}.
By sliding the arc $K'$ on $K'' \cup \gamma_2$, 
we may have the other examples. 

\begin{figure}[htbp]
\centering
\includegraphics[width=.4\textwidth]{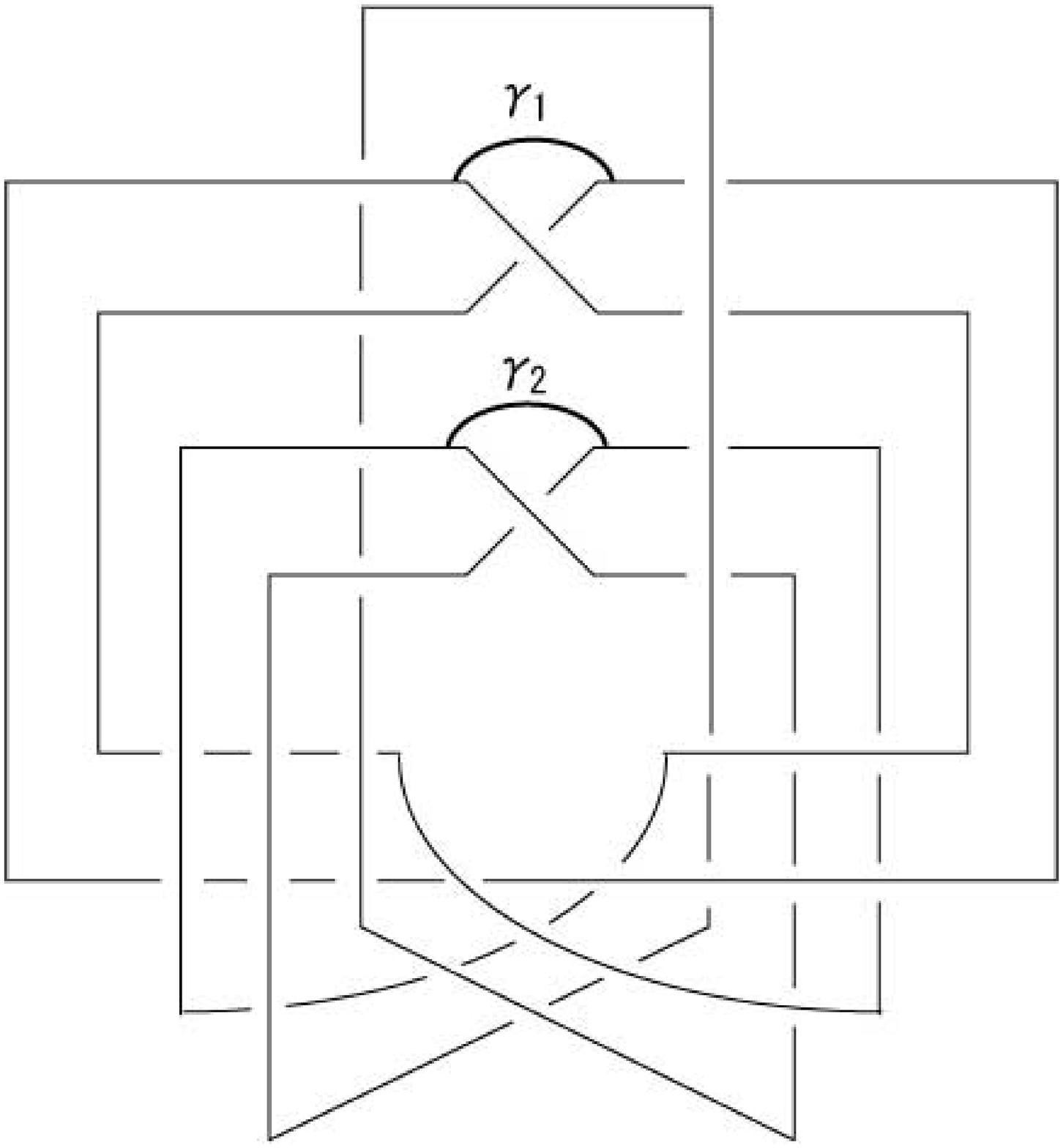}
\caption{}
\label{fig:572}
\end{figure}

\begin{figure}[htbp]
\centering
\includegraphics[width=.4\textwidth]{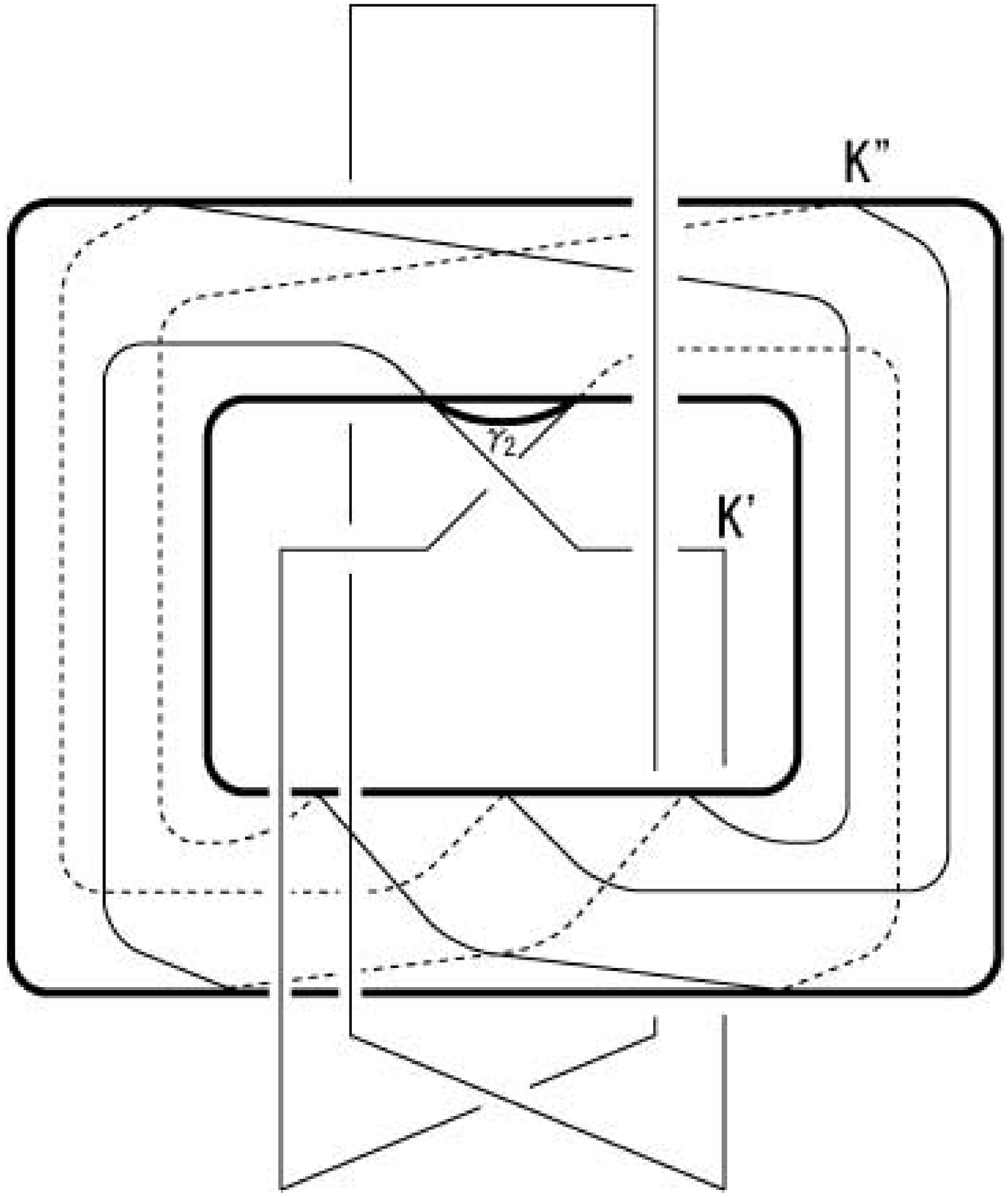}
\caption{}
\label{fig:572torus}
\end{figure}

\bibliographystyle{amsplain}

\begin{thebibliography}{30}

\bibitem{Bo} F. Bonahon,
\textit{Diff\'{e}otopies des espaces lenticulaires},
Topology 22 (1983), 305--314.

\bibitem{BoO} F. Bonahon and J.P. Otal,
\textit{Scindements de Heegaard des espaces lenticulaires},
Ann. Sci. Ec. Norm. Super. 16 (1983), 451-466.

\bibitem{BRZ} Z. Boileau, M. Rost and H. Zieschang,
\textit{On Heegaard decompositions of torus exteriors
and related Seifert fibred spaces},
Math. Ann. 279 (1988), 553--581.

\bibitem{CM1} S. Cho and D. McCullough, 
\textit{The tree of knot tunnels}, 
Geom. Topol. 13 (2009), 769--815.

\bibitem{CM2} S. Cho and D. McCullough, 
\textit{The depth of a knot tunnel}, 
preprint, arXiv:0708.3399. 

\bibitem{GH} H. Goda and C. Hayashi,
\textit{Genus two Heegaard splittings of exteriors
of 1-genus 1-bridge knots II},
preprint.

\bibitem{GHY} H. Goda, C. Hayashi and N. Yoshida,
\textit{Genus two Heegaard splittings of exteriors of knots
and the disjoint curve property},
Kobe J. Math. 18 (2001), 79--114.

\bibitem{GST} H. Goda, M. Scharlemann and A. Thompson,
\textit{Levelling an unknotting tunnel},
Geom. Topol. 4 (2000), 243-275.


\bibitem{Hy1} C. Hayashi,
\textit{Genus one $1$-bridge positions for the trivial knot and cabled
knots},
Math. Proc. Camb. Philos. Soc. 125 (1999), 53--65.

\bibitem{Hy2} C. Hayashi,
\textit{Satellite knots in $1$-genus $1$-bridge positions},
Osaka J. Math. 36 (1999), 203--221.

\bibitem{Hy3} C. Hayashi,
\textit{1-genus 1-bridge splittings for knots},
Osaka J. Math. 41 (2004), 371--426.

\bibitem{I} K. Ishihara, 
\textit{Algorithm for finding parameter of tunnels}, 
preprint.

\bibitem{johnson} J. Johnson, 
\textit{Bridge Number and the Curve Complex}, 
preprint, arXiv:math/0603102.

\bibitem{K0} T. Kobayashi,
\textit{Structures of the Haken manifolds
with Heegaard splitting of genus two},
Osaka J. Math. 21 (1984), 437-455.

\bibitem{K1} T. Kobayashi,
\textit{Classification of unknotting tunnels for two bridge knots},
Proceedings of the Kirbyfest,
Geom. Topol. Monogr. 2 (1999), 259-290.

\bibitem{K2} T. Kobayashi,
\textit{Heegaard splittings of exteriors of two bridge knots},
Geom. Topol. 5 (2001), 609--650.

\bibitem{KS} T. Kobayashi and O. Saeki,
\textit{Rubinstein-Scharlemann graphic of 3-manifold
as the descriminant set of a stable map},
Pac. J. Math. 195 (2000), 101-156.

\bibitem{koda} Y. Koda, 
\textit{Tunnel complexes of 3-manifolds}, 
preprint.

\bibitem{M} K. Morimoto,
\textit{On minimum genus Heegaard splittings
of some orientable closed $3$-manifolds},
Tokyo J. Math. 12 (1989), 321--355.

\bibitem{MS} K. Morimoto and M. Sakuma,
\textit{On unknotting tunnels for knots},
Math. Ann. 289 (1991), 143-167.

\bibitem{MSY} K. Morimoto, M. Sakuma, Y. Yokota,
\textit{Examples of tunnel number one knots 
which have the property $\lq$1+1=3'},
Math. Proc. Camb. Philos. Soc. 119 (1996), 113-118.

\bibitem{N} F. H. Norwood, 
\textit{Every two-generator knot is prime}, 
Proc. Amer. Math. Soc. 86 (1982), 143--147.

\bibitem{RS} H. Rubinstein and M. Scharlemann,
\textit{Comparing Heegaard splittings of non-Haken 3-manifolds},
Topology 35 (1996), 1005-1026.

\bibitem{S} M. Scharlemann, 
\textit{Tunnel number one knots satisfy the Poenaru conjecture}, 
Topology Appl. 18 (1984), 235--258.

\bibitem{ST} M. Scharlemann and M. Tomova, 
\textit{Alternate Heegaard genus bounds distance},
Geom. Topol. 10 (2006), 593--617.
\end{thebibliography}

\end{document}